\documentclass{amsart}
\usepackage[latin1]{inputenc}
\usepackage{amssymb}
\usepackage{amsmath}
\usepackage{enumerate}
\usepackage{epsfig,color,xcolor}

\usepackage[all]{xy}

\newtheorem{theorem}{Theorem}[section]
\newtheorem{lemma}[theorem]{Lemma}
\newtheorem{proposition}[theorem]{Proposition}
\newtheorem{corollary}[theorem]{Corollary}

\newtheorem{definition}[theorem]{Definition}

\newtheorem{rem}[theorem]{Remark}
\newtheorem{remark}[theorem]{Remark}

\numberwithin{equation}{section}
\newcommand{\rk}{{\rm rank}}

\DeclareMathOperator{\NS}{NS}
\newcommand{\Pic}{\mbox{Pic}}

\newcommand{\ra}{\rightarrow}
\newcommand{\dra}{\dashrightarrow}

\newcommand{\diag}{{\rm diag}}

\newcommand{\Z}{\mathbb{Z}}
\newcommand{\Q}{\mathbb{Q}}

\newcommand{\N}{\mathbb{N}}

\newcommand{\cP}{\mathcal{P}}

\newcommand{\Aut}{\mbox{Aut}}

\makeatletter
\def\blfootnote{\xdef\@thefnmark{}\@footnotetext}

\title{On certain isogenies between K3 surfaces}
\author{Chiara Camere}
\author{Alice Garbagnati}
\address{Chiara Camere,
 Dipartimento di Matematica,
Universit\`a degli Studi di Genova, Via Dodecaneso 35, 16146 Genova (GE) }
\email{camere@dima.unige.it}

\address{Alice Garbagnati, Universit\`a degli Studi di Milano, Dipartimento di Matematica, via Cesare Saldini 50 20133 Milano, Italy }
\email{alice.garbagnati@unimi.it}
\urladdr{https://sites.google.com/site/alicegarbagnati/}


\begin{document}

\begin{abstract}
The aim of this paper is to construct ``special" isogenies between K3 surfaces, which are not Galois covers between K3 surfaces, but are obtained by composing cyclic Galois covers, induced by quotients by symplectic automorphisms. We  determine the families of K3 surfaces for which this construction is possible.  To this purpose we will prove that there are infinitely many big families of K3 surfaces which both admit a finite symplectic automorphism and are (desingularizations of) quotients of other K3 surfaces by a symplectic automorphism.

In the case of involutions, for any $n\in\mathbb{N}_{>0}$ we determine the transcendental lattices of the K3 surfaces which are $2^n:1$ isogenous (by a non Galois cover) to other K3 surfaces. We also study the Galois closure of the $2^2:1$ isogenies and we describe the explicit geometry on an example.\end{abstract}
\maketitle

\section{Introduction}\label{intro}
K3 surfaces are symplectic regular surfaces and among their
finite order automorphisms the ones which preserve the symplectic
structure (the symplectic automorphisms) play a special role.
Indeed, the quotient of a K3 surface by a finite symplectic
automorphism produces a singular surface whose desingularization
is again a K3 surface. This construction establishes a particular
relation between different sets of K3 surfaces: the ones which
admit a finite symplectic automorphism and the ones obtained as desingularization of the quotient of a K3
surfaces by a symplectic automorphism. In the following the latter
K3 surfaces are said to be (cyclically) covered by a K3 surface
and the former are said to be the cover of a K3 surface. We denote by $\mathcal{L}_n$ the set of the K3
surfaces which admit an order $n$ symplectic automorphism and by $\mathcal{M}_n$ the set of the K3 surfaces which are $n:1$ cyclically covered by a K3 surface. From now on
we assume the surfaces to be projective.

Thanks to several works, starting from the end of the 70's until
now (see, e.g. \cite{NikSympl}, \cite{Morrison}, \cite{vGS},
\cite{GS}, \cite{GSprime} \cite{GSnonprime}, \cite{G}), the sets $\mathcal{L}_n$ and $\mathcal{M}_n$ are described as the union of countably many families of $R$ polarized K3 surfaces, for certain known lattices $R$. The dimension of these families is at most 11, and, recalling that the families of generic projective K3 surfaces have dimension 19, one immediately observes that the K3 surfaces which either admit a finite symplectic automorphism or which are cyclically covered by a K3 surface are quite special. So, it is natural to expect that the
intersection $\mathcal{L}_n\cap\mathcal{M}_n$ is extremely small, i.e. that a K3 surface which is both covered and cover of another K3 surface is really rare. On the other hand, there is at least one known example of a
family of K3 surfaces contained in $\mathcal{L}_n\cap\mathcal{M}_n$, given by the family of the K3 surfaces which admit an
elliptic fibration with an $n$-torsion section (see Section \ref{sec:intersection}). This family has codimension one in the
families which are components of $\mathcal{L}_n$ and of
$\mathcal{M}_n$. Hence, surprisingly, the intersection $\mathcal{L}_n\cap\mathcal{M}_n$ is not so small.

The aim of this paper is to investigate more precisely the
intersection between the two sets $\mathcal{L}_n$ and
$\mathcal{M}_n$ and to relate it with the study of  isogenies
between K3 surfaces. In this paper, the term ``isogeny between K3
surface" means a generically finite rational map between K3
surfaces, as in \cite{I} and \cite{BSV}.

The quotient by a finite symplectic automorphism on a K3 surface $X$
induces an isogeny between $X$, which admits the
symplectic automorphism, and the K3 surface $Y$ cyclically covered
by $X$. The isogeny is birationally the quotient map and has of
course the same order as the automorphism. There are other
isogenies between K3 surfaces, which are not quotient maps, see
e.g. \cite{I} and \cite{BSV}.  Here we discuss one of these other
isogenies:  given a K3 surface $Z\in\mathcal{L}_n\cap\mathcal{M}_n$, it induces an $n^2:1$ isogeny between other two K3 surfaces. Indeed, since $Z\in\mathcal{M}_n$, it is $n:1$ covered by a K3 surface $X$; since $Z\in\mathcal{L}_n$, it is an $n:1$ cover of a K3 surface $Y$. By composing these two $n:1$ maps one obtains an $n^2:1$ isogeny between $X$ and $Y$. We will prove that generically this isogeny is not induced by a quotient map.

In Section \ref{sec: prlim results} we recall  some
preliminary results on the set $\mathcal{L}_n$ of K3 surfaces
admitting a symplectic automorphism of order $n$ and on the set
$\mathcal{M}_n$ of the K3 surfaces $n:1$ cyclically covered by a
K3 surface. In Section \ref{sec:intersection} we obtain our
main results on the intersection $\mathcal{L}_n\cap\mathcal{M}_n$. In particular in Theorem \ref{theorem: intersection Ln and Mn} we prove:

{\bf Theorem }{\it There
are components $\mathcal{Z}$ of $\mathcal{L}_n\cap\mathcal{M}_n$ such that $\dim\left(\mathcal{L}_n\right)=\dim\left(\mathcal{M}_n\right)=\dim \mathcal{Z}$, i.e. the dimension of $\mathcal{Z}$ is the maximal possible and thus $\mathcal{Z}$ is an irreducible component of both of $\mathcal{L}_n$ and $\mathcal{M}_n$.}

As a consequence we construct $n^2:1$ isogenies and we prove that
generically they are not quotient maps. The Section \ref{sec:involutions} contains the main results for the
case $n=2$. In addition to the results which hold for every
admissible $n$, we also obtain the following theorem (see Theorem \ref{theorem: intersection L2 and M2} and Corollary \ref{cor: infinite sets of families})

{\bf Theorem }{\it For any $d,n\in\mathbb{N}>0$, there exists a lattice $R_{d,n}$ (with $R_{d,n}\simeq R_{d',n'}$ if and only if $(d,n)=(d',n')$) and there exists a family of $R_{d,n}$-polarized K3 surfaces such that, for any $m\in\mathbb{N}_{>0}$ and any $R_{d,n}$-polarized K3 surface $X$ there exists an $R_{d,m}$-polarized K3 surface $Y$ isogenous to $X$ with an isogeny of degree $2^{|n-m|}$.

So for each $d\in\mathbb{N}_{>0}$ there are countably many families
of polarized K3 surfaces, such that there exists an isogeny
between members of each family.}

The N\'eron--Severi group and the transcendental lattice of all the surfaces involved in these
isogenies are explicitly given. In Section \ref{subsec: Galois  closure 2^2:1 cover} we describe the Galois closure of the $2^2:1$ (non Galois) covers constructed. Moreover, in Section \ref{sec: X1 and Y2} we describe the geometry of a generic member $X_2$ of a certain maximal dimensional family of K3 surfaces which is contained in $\mathcal{L}_2\cap\mathcal{M}_2$.
The K3 surface $X_2$ admits two different polarizations of degree
4: one exhibits the surface $X_2$ as a special singular quartic in
$\mathbb{P}^3$ with eight nodes, the other as smooth double cover
of a quadric in $\mathbb{P}^3$. The former model is the singular
quotient of another K3 surface by a symplectic involution (thus it
implies that $X_2\in\mathcal{M}_2$), the latter implies that $X_2$
admits a symplectic involution induced by the switching of the
rulings on the quadric (thus it implies that
$X_2\in\mathcal{L}_2$). We describe both projective models of $X_2$ and give the
explicit relation between them, providing a geometric realization
of the previous lattice theoretic result which guarantees that
$X_2$ is both covered by a K3 surface and is a cover of a K3 surface. In particular this allows us to describe a symplectic involution on
the model of $X_2$ as singular quotient. In Section \ref{subsec: U+N e U+E8} we
analyse the similar problem for two specific families of
codimension 1 in $\mathcal{L}_2\cap\mathcal{M}_2$; one of these
families is totally contained in all the components of
$\mathcal{L}_2$ and the other in all the components of
$\mathcal{M}_2$.

\subsection*{Acknowledgements} Part of this project was realized while the first-named author was visiting the Max Planck Institute of Mathematics in Bonn: the first-named author is grateful to Max Planck Institute for Mathematics in Bonn for its hospitality and financial support.
We are grateful to Bert van Geemen for his enlightening suggestions and to Simon Brandhorst for his precious remarks.

\section{Preliminary results}\label{sec: prlim results}
We recall in this section some of the definitions and results on K3 surfaces, symplectic automorphisms on K3 surfaces and quotients of K3 surfaces by their automorphisms.
In the following we work with projective surfaces.

\subsection{Symplectic automorphisms and cyclic covers of K3 surfaces}

\begin{definition}
A (projective) K3 surface is a regular projective surface with trivial canonical bundle. If $X$ is a K3 surface, we choose a generator of $H^{2,0}(X)$, (i.e. a symplectic form), we denote it by $\omega_X$ and we call it the period of the K3 surface.
The second cohomology group $ H^2(X,\Z)$ of a K3 surface $X$ equipped with the cup product is a lattice, isometric to a standard lattice which does not depend on $X$ and is denoted by $\Lambda_{K3}:= U^{\oplus 3}\oplus E_8(-1)^{\oplus 2}$.
\end{definition}
\begin{definition}
Let $X$ be a K3 surface, and $\omega_X$ its period. An automorphism $\sigma$ of $X$ is said to be symplectic if $\sigma^*(\omega_X)=\omega_X$.
\end{definition}
One of the main results on symplectic automorphisms on K3 surfaces
is that the quotient of a K3 surface by a symplectic automorphism
is still a K3 surface, after a birational transformation which
resolves the singularities of the surface.
\begin{proposition}{\rm (\cite{NikSympl})}
Let $X$ be a K3 surface and $\sigma\in \Aut(X)$ a finite automorphism of $X$. Then the minimal smooth surface $Y$ birational to $X/\sigma$ is a K3 surface if and only if $\sigma$ is symplectic.
\end{proposition}
\begin{definition}
We will say that a K3 surface $Y$ is $n:1$ cyclically covered by a
K3 surface, if there exists a pair $(X,\sigma)$ such that $X$ is a
K3 surface, $\sigma$ is an automorphism of order $n$ of $X$ and
$Y$ is birational to $X/\sigma$.
\end{definition}

The first mathematician who worked on symplectic automorphisms of finite order on K3 surfaces and who established the fundamental results
on these automorphisms was Nikulin, in \cite{NikSympl}. We
summarize in Theorem \ref{therem: Nikulin results} and Theorem \ref{theorem: X has sympl autom iff Omega is in  NS(X)} the main
results obtained in his paper, but first we recall some useful
information and definitions.

If $\sigma$ is a symplectic automorphism on $X$ of order $n$, its
linearization near the points with non trivial stabilizer is given
by a $2\times 2$ diagonal matrix with determinant 1 and thus it is of the form $\diag(\zeta_n^a,\zeta_n^{n-a})$ for
$1\leq a\leq n-1$ and $\zeta_n$ an $n$-th primitive root of unity.
So, the points with non trivial stabilizer are isolated fixed
points and the quotient $X/\sigma$ has isolated singularities, all
of type $A_{m_j}$ where $m_j+1$ divides $n$. In particular the
surface $Y$, which is the minimal surface resolving the
singularities of $X/\sigma$, contains smooth rational curves $M_i$
arising from the desingularization of $X/\sigma$. The classes of
these curves span a lattice isometric to $\oplus_{j} A_{m_j}$.

\begin{definition}
Let $Y$ be a K3 surface, $n:1$ cyclically covered by a K3 surface.
The minimal primitive sublattice of $\NS(Y)$ containing the classes
of the curves $M_i$, arising from the desingularization of
$X/\sigma$, is denoted by $\mathbb{M}_n$.
\end{definition}
We observe that $\mathbb{M}_n$ is necessarily an overlattice of
finite index (a priori possibly 1) of the lattice $\oplus_{j}
A_{m_j}$ spanned by the curves $M_i$. The presence of a smooth
cyclic cover of $X/\sigma$ branched over the singular points
obtained as contraction of the curves $M_i$ suggests that there
are some divisibility relations among the $M_i$'s and thus that the
index of the inclusion $\left\langle \left(M_i\right)_i\right\rangle\hookrightarrow \mathbb{M}_n$
would not be 1 (as indeed stated in Theorem \ref{therem: Nikulin
results}).

\begin{definition}{\rm (See \cite[Definition 4.6]{NikSympl})}\label{defi: essentially unique action}
Let $\sigma$ be an order $n$ automorphism of a K3 surface $X$. We will say that its action on the second cohomology group is essentially unique if there exists an isometry $g_n:\Lambda_{K3}\stackrel{\sim}{\longrightarrow}\Lambda_{K3}$ of order $n$ of $\Lambda_{K3}$ such that for every pair $(X,\sigma)$, there exists an isometry $\varphi:H^2(X,\Z)\ra \Lambda_{K3}$ such that $\sigma^*=\varphi^{-1}\circ g_n\circ\varphi$.
\end{definition}

\begin{theorem}\label{therem: Nikulin results}
Let $X$ be a K3 surface and $\sigma$ a finite symplectic automorphism of $X$ of order $|\sigma|=n$. Then
\begin{itemize}
\item $2\leq n\leq 8$ (see \cite[Theorem 6.3]{NikSympl}); \item the
singularities of $X/\sigma$ depend only on $n$ (see
\cite[Section 5]{NikSympl}); \item the class of isometry of the lattice
$\mathbb{M}_n$ depends only on $n$ and
$\mathbb{M}_n$ is an overlattice of index $n$ of the lattice
$\left\langle\left( M_i\right)_i\right\rangle$ spanned by the curves arising from the desingularization
of the quotient $X/\sigma$  (see \cite[Theorem 6.3]{NikSympl}); \item the action of $\sigma^*$ on
$H^2(X,\Z)$ is essentially unique (see \cite[Theorem 4.7]{NikSympl})
and thus the classes of isometry of the lattices $H^2(X,\Z)^{\sigma^*}$ and
$\left(H^2(X,\Z)^{\sigma^*}\right)^{\perp}$ depend only on $n$.
\item The lattice $\left(H^2(X,\Z)^{\sigma}\right)^{\perp}$ is primitively embedded in $\NS(X)$ (see \cite[Lemma 4.2]{NikSympl}) and $\rk(\left(\Lambda_{K3}^{g_n}\right)^{\perp})=\rk(\mathbb{M}_n)$ (see \cite[Formula (8.12)]{NikSympl}).
\end{itemize}
\end{theorem}

\begin{definition}
Let $X$ be a K3 surface with a symplectic automorphism $\sigma$ of order $n$. Since the action of $\sigma^*$ on $H^2(X,\Z)$ is essentially unique, the lattice $\left(H^2(X,\Z)^{\sigma^*}\right)^{\perp}$ is isometric to $\left(\Lambda_{K3}^{g_n}\right)^{\perp}$ (with the notation of Definition \ref{defi: essentially unique action}) and we denote it by $\Omega_{n}$.
\end{definition}

For every admissible $n$ the lattices $\Omega_n$ were computed: in \cite{vGS} and \cite{Morrison} if $n=2$; in \cite{GSprime} if $n$ is an odd prime; in \cite{GSnonprime} if $n$ is not a prime.

The lattices $\mathbb{M}_n$ were computed for every admissible $n$ in \cite[Theorems 6.3 and 7.1]{NikSympl}.

The lattices $\Omega_n$ and $\mathbb{M}_n$ characterize the K3 surfaces admitting a symplectic automorphism of order $n$ or a $n:1$ cyclic cover by a K3 surface respectively; indeed, the following two results hold
\begin{theorem}\label{theorem: X has sympl autom iff Omega is in NS(X)}{\rm (See \cite[Theorem 4.15]{NikSympl})}
A K3 surface $X$ admits a symplectic automorphism of order $n$ if and only if $\Omega_n$ is primitively embedded in $\NS(X)$.
\end{theorem}
\begin{theorem}\label{theorem: Y has cyclic cover iff Mn is in NS(Y)}{\rm(See \cite[Proposition 2.3]{GS} for the case $n=2$ and \cite[Theorem 5.2]{G} for other $n$)}
A K3 surface $Y$ is $n:1$ cyclically covered by a K3 surface if and only if $\mathbb{M}_n$ is primitively embedded in $\NS(Y)$.
\end{theorem}

\begin{corollary}\label{corollary: NS of K3 in Ln and Mn}
Let $X$ be a projective K3 surface admitting a symplectic automorphism of order $n$. Then $\rho(X)\geq 1+\rk(\Omega_n)$ and if $\rho(X)=1+\rk(\Omega_n)$, then $\NS(X)$ is an overlattice of finite index (possibly 1) of $\langle 2d\rangle\oplus \Omega_n$, for a certain $d\in\mathbb{N}_{>0}$, such that $\Omega_n$ is primitively embedded in this overlattice.

Let $Y$ be a projective K3 surface $n:1$ cyclically cover by a K3 surface. Then $\rho(Y)\geq 1+\rk(\mathbb{M}_n)$ and if $\rho(Y)=1+\rk(\mathbb{M}_n)$, then $\NS(Y)$ is an overlattice of finite index (possibly 1) of $\langle 2e\rangle\oplus \mathbb{M}_n$, for a certain $e\in\mathbb{N}_{>0}$, such that $\mathbb{M}_n$ is primitively embedded in this overlattice.
\end{corollary}
\proof Since $X$ admits a symplectic automorphism of order $n$, $\Omega_n$ is primitively embedded in $\NS(X)$. Since $\Omega_n$ is negative definite and $X$ is projective, the orthogonal to $\Omega_n$ in $\NS(X)$ contains a class with a positive self intersection, in particular it is non empty. So $\rho(X)\geq 1+\rk(\Omega_n)$ and $\langle 2d\rangle\oplus \Omega_n$ is embedded in $\NS(X)$. Similarly one obtains the result for $\rho(Y)$ and $\NS(Y)$.\endproof

\begin{definition}
We define the following sets of K3 surfaces (which are subsets of the moduli space of the K3 surfaces):
$$\mathcal{L}_n:=\{\mbox{K3 surfaces which admit a symplectic automorphims }\sigma\mbox{ of order }n\}/\cong,$$
$$\mathcal{M}_n:=\{\mbox{K3 surfaces which admit an }n:1\mbox{ cyclic cover by a K3 surface}\}/\cong,$$
where $\cong$ denotes the equivalence relation given by isomorphism between two K3 surfaces.
\end{definition}

Given an even hyperbolic lattice $R$ which admits a primitive embedding in $\Lambda_{K3}$, we denote by $\mathcal{P}(R)$ the moduli space of isomorphism classes of $R$-polarized $K3$ surfaces, i.e. of those $K3$ surfaces $X$ for which there exists a primitive embedding $R\subset \NS(X)$. Moreover, we will write $A<B$ in order to say that $B$ is an overlattice of finite index of $A$.

\begin{corollary}\label{corollary construction sets Ln and Mn}
The set $\mathcal{L}_n$ is a union of countably many components and  each of them is a family of $R$-polarized K3 surfaces, for an appropriate choice of the lattice $R$:
\[
\mathcal{L}_n=\bigcup_{d\in\N}\left(\bigcup_{\substack{(\langle 2d\rangle\oplus\Omega_n)< R\\ \Omega_n\subset R\ \mathrm{prim.}}}\mathcal{P}(R)\right).
\]
All the components $\mathcal{P}(R)$ are equidimensional and have dimension $19-\rk(\Omega_n)$.

The set $\mathcal{M}_n$ is a union of countably many components and  each of them is a family of $R$-polarized K3 surfaces, for an appropriate choice of the lattice $R$:
\[
\mathcal{M}_n=\bigcup_{d\in\N}\left(\bigcup_{\substack{(\langle 2d\rangle\oplus\mathbb{M}_n)< R\\ \mathbb{M}_n\subset R\ \mathrm{prim.}}}\mathcal{P}(R)\right).
\]

All the components are equidimensional and have dimension $19-\rk(\mathbb{M}_n)=19-\rk(\Omega_n)$.
\end{corollary}
\proof Let $R$ be an overlattice of finite index of $\langle 2d\rangle\oplus\Omega_n$ such that $\Omega_n$ is primitively embedded in it. If $X$ is a K3 surface such that $R$ is primitively embedded in $\NS(X)$, then $\Omega_n$ is primitively embedded in $\NS(X)$ and thus $X$ admits a symplectic automorphism of order $n$,  by Theorem \ref{theorem: X has sympl autom iff Omega is in NS(X)}.  Vice versa, if a projective K3 surface $X$ admits a symplectic automorphism of order $n$, then there exists a $d\in\mathbb{N}>0$ such that $\langle 2d\rangle\oplus \Omega_n$ is embedded in $\NS(X)$, and an overlattice $R$ of $\langle 2d\rangle\oplus \Omega_n$ is primitively embedded in $\NS(X)$. So one can describe the set $\mathcal{L}_n$ as union of families $\cP(R)$ of $R$-polarized K3 surfaces, where $R$ is a proper overlattice of index $r$ (possibly 1) of $\langle 2d\rangle\oplus \Omega_n$ for a certain $d\in\mathbb{N}$. There are countably many lattices $\langle 2d\rangle\oplus\Omega_n$ and each of them has a finite number of overlattices of finite index. So $\mathcal{L}_n$ is the union of countably many families of $R$-polarized K3 surfaces. The dimension of each of these families is $20-\rk(R)=20-(1+\rk(\Omega_n))$. This concludes the proof for the set $\mathcal{L}_n$.

The proof for $\mathcal{M}_n$ is similar, but one has to use the Theorem \ref{theorem: Y has cyclic cover iff Mn is in  NS(Y)} instead of the Theorem \ref{theorem: X has sympl autom iff Omega is in  NS(X)}.
\endproof

\subsection{Isogenies between K3 surfaces}

The following definition was first given by Inose in \cite{I} in the case of K3 surfaces with Picard number $20$.

\begin{definition}\label{def: isogenies}
Let $X$ and $Y$ be two K3 surfaces. We say that $X$ and $Y$ are
isogenous if there exists a rational map of finite degree between $X$ and
$Y$. This map is said to be an isogeny between $X$ and $Y$ and if
it is generically of degree $n$, the map is said to be an isogeny
of degree $n$.
\end{definition}

The easiest construction of an isogeny between K3 surfaces is given
by the quotient by a finite symplectic automorphism, i.e. if $X$
is a K3 surface admitting a symplectic automorphism $\sigma$ of order $n$,
then the quotient map induces an isogeny of degree $n$ between $X$
and $Y$, the minimal model of $X/\sigma$. So if
$X\in\mathcal{L}_n$, then there exists
$Y\in\mathcal{M}_n$ which is isogenous to $X$ with an isogeny of
degree $n$. Similarly if $Y\in\mathcal{M}_n$, then there exists a
K3 surface $X\in\mathcal{L}_n$ which is isogenous to $Y$ with an
isogeny $X\dra Y$ of degree $n$.

There exist however isogenies between K3 surfaces which are not induced by
the quotient by a finite group of symplectic automorphisms:
an example is given by isogenous Kummer surfaces constructed from Abelian surfaces related by an isogeny, as in \cite[Proof of Thm 2]{I}, under the additional assumption that the degree is a prime $p>7$, (see also \cite[Example 6.5]{BSV}).

Now, let us suppose that $Z$ is a K3 surface such that $Z\in
\mathcal{L}_n\cap \mathcal{M}_n$. Then, there exists a K3 surface
$X\in\mathcal{L}_n$ which is isogenous to $Z$, with an isogeny
$\rho:X\dra Z$ of degree $n$, but also a K3 surface $Y\in
\mathcal{M}_n$ which is isogenous to $Z$ with an isogeny
$\pi:Z\dra Y$ of degree $n$. So the existence of
$Z\in\mathcal{L}_n\cap\mathcal{M}_n$ allows one to construct an
isogeny of degree $n^2$ between the two K3 surfaces $X$ and $Y$,
given by the composition $\pi\circ\rho:X\dra Y$. We will show that
in many cases this isogeny is not induced by a quotient by a
finite group of symplectic automorphisms acting on $X$, see
Proposition \ref{prop: n^2 isogenies not quotients}.

In the Section \ref{sec:intersection} we prove that
$\mathcal{L}_n\cap \mathcal{M}_n$ is non empty if $2\leq n\leq 8$
and then we provide examples of $n^2:1$ isogenies between K3
surfaces.

\subsection{Remarks on Hodge isogenies between K3 surfaces}

The Definition \ref{def: isogenies} is not the only notion of
isogeny existing in the literature: to distinguish between the two
definitions, we will talk here of {\it Hodge isogeny} for the
notion used for example in \cite{Buskin,Huybrechts}.

\begin{definition}
Let $X$ and $Y$ be two K3 surfaces. We say the $X$ and $Y$ are Hodge
isogenous if there exists a rational Hodge isometry between $H^2(X,\Q)$ and $H^2(Y,\Q)$.
\end{definition}

Hodge isogenous K3 surfaces have been studied since foundational work of \cite{Mukai} and \cite{NikulinCorresp}, also in relation with \v{S}afarevi\v{c}'s conjecture \cite{Shafarevich} about algebraicity of correspondences on $K3$ surfaces.

In \cite[Prop. 3.1]{BSV}, the authors give a comparison between the notion of isogeny and of Hodge isogeny:

\begin{proposition}\label{prop:BSV}
If $\varphi: X\dashrightarrow Y$ is an isogeny of order $n$, $n$ is not a square and the rank of the transcendental lattices $T_X$ and $T_Y$ is odd, $\varphi$ is never a Hodge isogeny.
\end{proposition}

This follows from the fact that, under these assumptions, there cannot exist any isometry $T_X\otimes \Q\simeq T_Y\otimes \Q$. The transcendental lattice $T_X$ of the very general K3 surface $X\in\mathcal{L}_n$ has always odd rank (see Theorem \ref{theorem: intersection Ln and Mn}); by Proposition \ref{prop:BSV} if $n$ is not a square, so if $n\neq 4$, the surface $X$ is never Hodge isogenous to the minimal resolution of its quotient. The assumption on the degree $n$ is in particular due to the following straightforward fact:

\begin{lemma}\label{lemma:isometry T and T(n^2)}
 For any non degenerate lattice $T$ and any integer $n\in\N$, there exists an isometry $T\otimes \Q\simeq T(n^2)\otimes \Q$.
\end{lemma}

\begin{proposition}\label{prop: Hodge isog}
For any $n\in\N$, if $\varphi:X\dra Y$ is an isogeny of degree $n^2$, then $X$ and $Y$ are Hodge isogenous.
\end{proposition}
\begin{proof} It is proven in \cite[Proposition 3.2]{BSV} that $T_X\otimes\Q\simeq T_Y\otimes\Q$ if and only if $T_Y\otimes \Q\simeq T_Y(n^2)\otimes \Q$, which is true by Lemma \ref{lemma:isometry T and T(n^2)}. Then Witt's theorem implies that the isometry $T_X\otimes\Q\simeq T_Y\otimes\Q$ extends to a Hodge isometry $H^2(X,\Q)\simeq H^2(Y,\Q)$.
\end{proof}

In Proposition \ref{prop: n^2 isogenies not quotients} we construct isogenies of degree $n^2$ between K3 surfaces;  Proposition \ref{prop: Hodge isog} implies that they are necessarily Hodge isogenies.

One of the interesting properties of Hodge isogenous K3 surfaces is that they have isomorphic rational motives, by \cite[Theorem 0.2]{Huybrechts}. This also holds in the case described above of a K3 surface $X$ isogenous to the minimal model $Y$ of the quotient $X/\sigma$, as shown for example in  \cite[Proof of Thm 3.1]{Laterveer} following the argument of \cite{P}, but to the knowledge of the authors it is still an open question for a general isogeny.

\section{The intersection $\mathcal{L}_n\cap\mathcal{M}_n$}\label{sec: intersection LnMn}\label{sec:intersection}

The main result in this section is Theorem \ref{theorem: intersection Ln and Mn}, where we exhibit the maximal dimensional components of $\mathcal{L}_n\cap \mathcal{M}_n$. As preliminary result, we describe in \S \ref{prop: properties of the U+Mn polarized K3} a specific family of K3 surfaces contained in  $\mathcal{L}_n\cap \mathcal{M}_n$. This family is related with a special isogeny between K3 surfaces, which is induced by an isogeny between elliptic curves, see Remark \ref{rem: isogeny between elliptic curve}.
\subsection{The $(U\oplus \mathbb{M}_n)$-polarized K3 surfaces}

The $(U\oplus \mathbb{M}_n)$-polarized K3 surfaces have interesting geometric properties: this family is considered for $n=2$ in \cite{vGS}, and for other values of $n$ in \cite{GSprime} and \cite{GSnonprime} to find explicitly $\Omega_n$. Here we reconsider it as example of a family of K3 surfaces contained in $\mathcal{L}_n\cap \mathcal{M}_n$.
\begin{proposition}\label{prop: properties of the U+Mn polarized K3}
Let $2\leq n\leq 8$ and $\mathcal{U}_n:=\cP(U\oplus  \mathbb{M}_n)$ be the family of the
$(U\oplus \mathbb{M}_n)$-polarized K3 surfaces.
Then:\begin{itemize}\item $\mathcal{U}_n$ is non empty and has dimension $18-\rk(\mathbb{M}_n)$;  \item if $S$ is a K3 surface such that
$S\in\mathcal{U}_n$, then $S$ admits an elliptic fibration
$\mathcal{E}_n:S\ra\mathbb{P}^1$ with an $n$-torsion section $t$;
\item $\mathcal{U}_n\subset \mathcal{L}_n\cap \mathcal{M}_n$; \item
if $S\in\mathcal{U}_n$ and $\sigma_t$ is the translation by $t$ on
$\mathcal{E}_n$, the minimal model of $S/\sigma_t$ is a K3 surface in
$\mathcal{U}_n$.\end{itemize}
\end{proposition}
\proof  The family $\mathcal{U}_n$ is
non empty for each $n$ such that $2\leq n\leq 8$ as showed for example in \cite[Table 2]{Shim} or \cite[Table 1]{G2}. The dimension of $\mathcal{U}_n$ follows directly by the fact that the dimension of a non-empty family $\cP(R)$ of $R$-polarized K3 surfaces (for a certain lattice $R$) is $20-\rk(R)$, and in this case $R\simeq U\oplus \mathbb{M}_n$ has rank $2+\rk(\mathbb{M}_n)$. The family $\mathcal{U}_n$ was considered in \cite[Proposition 4.3]{G2}, where it is proved that the set of K3 surfaces admitting an elliptic fibration with a torsion section of order $n$ coincides with the set of $(U\oplus \mathbb{M}_n)$-polarized K3 surfaces. Since $\mathbb{M}_n$ is clearly primitively embedded in $U\oplus \mathbb{M}_n$, all the K3 surfaces in $\mathcal{U}_n$ are also contained in $\mathcal{M}_n$. Moreover, let $\mathcal{E}_n:S\ra\mathbb{P}^1$ be an elliptic fibration on $S$ with an $n$-torsion section $t$. This allows to consider $S$ as an elliptic curve over the field of functions $k(\mathbb{P}^1)$ and the presence of an $n$-torsion section is equivalent to the presence of an $n$-torsion rational point on this elliptic curve. So the translation by $t$ is well defined and it induces an automorphism of order $n$ on $S$. This is a symplectic automorphism (it is the identity on the base of the fibration and acts on the smooth fibers preserving their periods). We denote this symplectic automorphism by $\sigma_t$.
Since $S\in\mathcal{U}_n$ admits a symplectic automorphism of
order $n$, $\mathcal{U}_n\subset\mathcal{L}_n$ and thus
$\mathcal{U}_n\subset\mathcal{L}_n\cap \mathcal{M}_n$. In \cite[Proposition 4.3]{G2}
 it is also proved that the quotient
of an elliptic fibration with basis $\mathbb{P}^1$ by the
translation by a torsion section is another elliptic fibration
over $\mathbb{P}^1$ with an $n$-torsion section. Thus $S/\sigma_t$
admits a smooth minimal model with an elliptic fibration with an
$n$-torsion section and this minimal model is a K3 surface (since
$\sigma_t$ is a symplectic automorphism). Thus the
minimal model of $S/\sigma_t$ belongs to the family $\mathcal{U}_n$.\endproof

\begin{corollary}\label{cor:non-empty intersection of Ln and Mn} For every $n$ such that $2\leq n\leq 8$,
$\mathcal{L}_n\cap\mathcal{M}_n$ is non empty.
\end{corollary}
\proof The intersection $\mathcal{L}_n\cap\mathcal{M}_n$ contains
at least the non empty family $\mathcal{U}_n$.\endproof
\begin{remark}\label{rem: isogeny between elliptic curve}{\rm Since both $\mathcal{L}_n$ and $\mathcal{M}_n$ are the union of $\left(19-\rk(\mathbb{M}_n)\right)$-dimensional families of polarized K3 surfaces, the intersection between these sets is at most $\left(19-\rk(\mathbb{M}_n)\right)$-dimensional. The Proposition \ref{prop: properties of the U+Mn polarized K3} provides an intersection in a codimension 1 subfamily. In the Theorem \ref{theorem: intersection Ln and Mn} we will see that one can obtain larger intersection.}
\end{remark}
\begin{remark}{\rm Since each $(U\oplus \mathbb{M}_n)$-polarized K3 surface has an elliptic fibration, it can be interpreted as elliptic curve over the field of rational functions in one variable. The symplectic automorphism which induces the isogeny between the two K3 surfaces in $\mathcal{U}_n$ as in Proposition \ref{prop: properties of the U+Mn polarized K3} is an isogeny of the associated elliptic curve over the field of rational functions.}
\end{remark}
\subsection{Maximal dimensional components of $\mathcal{L}_n\cap\mathcal{M}_n$}

In this section we prove that there are components of $\mathcal{L}_n$ completely contained in $\mathcal{M}_n$ and vice versa. The proof is lattice theoretic: in order to obtain this result, we need some extra information on the lattices $\mathbb{M}_n$ and $\Omega_n$. Both these lattices are primitively embedded in the N\'eron--Severi group of a $(U\oplus \mathbb{M}_n)$-polarized K3 surface, so we now use the K3 surfaces in the family $\mathcal{U}_n$ to compare the discriminant forms of $\mathbb{M}_n$ and $\Omega_n$.

\begin{proposition}\label{prop: discriminant Omega and Mn}
Let $A_{\Omega_n}$ (resp. $A_{\mathbb{M}_n}$) the discriminant group of $\Omega_n$  (resp. $\mathbb{M}_n$) and $q_{\Omega_n}$ (resp. $q_{\mathbb{M}_n}$) its discriminant form. Then $A_{\Omega_n}=\left(\Z/n\Z\right)^{\oplus 2}\oplus A_{\mathbb{M}_n}$ and  $q_{\Omega_n}=u(n)\oplus q_{\mathbb{M}_n}$, where $u(n)$ is the discriminant form of the lattice $U(n)$.
\end{proposition}
\proof Nikulin proved that $A_{\Omega_n}=\left(\Z/n\Z\right)^{\oplus 2}\oplus A_{\mathbb{M}_n}$ in \cite[Lemma 10.2]{NikSympl}.
By the Proposition \ref{prop: properties of the U+Mn
polarized K3}, if $\NS(S)\simeq U\oplus\mathbb{M}_n$, then $S$
admits an elliptic fibration $\mathcal{E}_n:S\ra\mathbb{P}^1$ with
an $n$-torsion section $t$ and thus a symplectic automorphism
$\sigma_t$, which is the translation by $t$. Let us denote by $F$
the class in $\NS(S)$ of the fiber of the elliptic fibration
$\mathcal{E}_n$, by $O$ the class of the zero section, by $t$ the
class of the $n$-torsion section, by $t_i$, $i=2,\ldots, n-1$, the
class of the section corresponding to the sum of $t$ with itself $i$
times in the Mordell--Weil group. By definition $\sigma_t$
preserves the classes $F$ and $O+t+\sum_{i=2}^{n-1}t_i$. So
$U(n)\simeq \langle F,O+t+\sum_{i=2}^{n-1}t_i\rangle\subset
\NS(S)^{\sigma_t}$ and thus $\langle
F,O+t+\sum_{i=2}^{n-1}t_i\rangle^{\perp}\supset
\left(\NS(S)^{\sigma_t}\right)^{\perp}.$ Since $\rk(\Omega_n)=\rk(\mathbb{M}_n)$, $\rk(\Omega_n)=\rho(S)-2$. So
$$\langle F,O+t+\sum_{i=2}^{n-1}t_i\rangle^{\perp}\supset
\left(\NS(S)^{\sigma_t}\right)^{\perp}\simeq\Omega_n.$$ Denoted by
$T_S$ the transcendental lattice of $S$, it follows from $\NS(S)\simeq U\oplus
\mathbb{M}_n$ that $q_{T_S}=-q_{\mathbb{M}_n}$. By
$\left(\NS(S)^{\sigma_t}\right)^{\perp}\simeq\Omega_n$ one obtains
that the orthogonal of $\Omega_n$ in $H^2(S,\Z)$ is an overlattice
of finite index (possibly 1) of $U(n)\oplus T_S$. Since $A_{\Omega_n}=\left(\Z/n\Z\right)^2\oplus A_{\mathbb{M}_n}$, the orthogonal of $\Omega_n$ in $H^2(S,\Z)$ is $U(n)\oplus
T_S$. So $q_{\Omega_n}\simeq -q_{U(n)\oplus T_S}=u(n)\oplus
q_{\mathbb{M}_n}$.\endproof

\begin{lemma}\label{lemma:discriminant form of overlattice}
Let  $F$ be a finite abelian group with quadratic form $q_F$ and $m\geq 2$.
Let $V=\langle 2d\rangle\oplus W$ be an indefinite even non-degenerate lattice with discriminant group $A_V=(\Z/2d\Z)\oplus (\Z/m\Z)^{\oplus 2}\oplus F$, with discriminant form $q_{A_V}=(\frac{1}{2d})\oplus u(m)\oplus q_F $. If $d\equiv 0\mod 2m$, then $V$ admits an overlattice $Z$ of index $m$ with $A_Z=\Z/2d\Z\oplus F$ and $q_{A_Z}=(\frac{1}{2d})\oplus  q_F $. Moreover, $Z$ contains primitively $W$.
\end{lemma}
\proof By assumption, there exists some integer $k$ such that
$2d=4km$. Let $h$ be a generator of the $\Z/2d\Z$ summand of $A_V$
such that $h^2=\frac{1}{2d}$, and let $e_1,e_2$ be a basis of the
$(\Z/m\Z)^{\oplus 2}$ summand in $A_V$ such that $e_1^2=e_2^2=0$
and $e_1e_2=-\frac{1}{m}$. We define $\epsilon:=e_1+2ke_2$, so
that $\epsilon^2=-\frac{4k}{m}$. Then the subgroup $H:=\langle
(4k)h+\epsilon\rangle$ is isotropic and its orthogonal inside
$A_V$ is $H^{\perp}=\langle h-e_2,\nu\rangle\oplus F$, where
$\nu:=e_1-2ke_2$. It follows from \cite[Propostion 1.4.1 and
Corollary 1.10.2]{NikulinIntQuadForms} that there exists an even
overlattice $Z$ of $V$ of index $m$ with $A_Z\cong H^{\perp}/H=
\langle h-e_2\rangle\oplus F\cong \Z/2d\Z\oplus F$, and $q_{A_Z}$
is induced on the quotient $H^{\perp}/H$ by
$(q_{A_V})_{|H^{\perp}}$, so it is exactly $(\frac{1}{2d})\oplus
q_F $. Finally, we observe that the intersection of $H$ with $A_W$
inside $A_V$ is trivial, hence $W$ is a primitive sublattice of
$Z$.\endproof

\begin{corollary}\label{cor: construction of Ldn'}
Let $2\leq n\leq 8$, $d\in\mathbb{N}$, $d\geq 1$ and $d\equiv 0\mod 2n$. Then
$\langle 2d\rangle\oplus\Omega_n$ admits an overlattice of index
$n$ whose discriminant form is $(\frac{1}{2d})\oplus
q_{\mathbb{M}_n}$; this overlattice contains primitively $\Omega_n$.

\end{corollary} \proof It suffices to apply Lemma
\ref{lemma:discriminant form of overlattice} to the lattice
$V=\langle 2d\rangle\oplus\Omega_n$ and to recall that
$q_{\Omega_n}=u(n)\oplus q_{\mathbb{M}_n}$, by Proposition
\ref{prop: discriminant Omega and Mn}.\endproof

\begin{definition}
For each $2\leq n\leq 8$ and each $d\in \mathbb{N}$, $d\geq 1$, we denote by $L_{d,n}$ the lattice $\langle 2d\rangle\oplus \Omega_n$.

For each $2\leq n\leq 8$ and each $d\in \mathbb{N}$, $d\geq 1$, and $d\equiv 0\mod 2n$, we denote by  $L_{d,n}'$ the overlattice of index $n$ of $L_{d,n}$ constructed in Corollary \ref{cor: construction of Ldn'}.

For each $2\leq n\leq 8$ and each $e\in \mathbb{N}$, $e\geq 1$, we denote by $M_{e,n}$ the lattice $\langle 2e\rangle\oplus \mathbb{M}_n$.
\end{definition}

\begin{theorem}\label{theorem: intersection Ln and Mn}
Let $d\in\mathbb{N}$, $d\geq 1$ and $d\equiv 0\mod 2n$. The lattice $L_{d,n}'$ is unique in its genus and $$L_{d,n}'\simeq M_{d,n}.$$ The family $\cP(L_{d,n}')$ is $\left(19-\rk(\Omega_n)\right)$-dimensional and is a subset of $\mathcal{L}_n\cap \mathcal{M}_n$, i.e. each K3 surface in this family admits a symplectic automorphism of order $n$ and is $n:1$ cyclically covered by a K3 surface.
\end{theorem}
\proof By the Corollary \ref{cor: construction of Ldn'}, the lattice $L_{d,n}'$ has the same discriminant group and form as the lattice $M_{d,n}$. By \cite[Proposition 7.1]{NikSympl}, the length and the rank of the lattice $\mathbb{M}_n$ are the following:
$$\begin{array}{|c|c|c|c|c|c|c|c|c|}
\hline
n&2&3&4&5&6&7&8\\
\hline
l(\mathbb{M}_n)&6&4&4&2&2&1&2\\
\hline
\rk(\mathbb{M}_n)&8&12&14&16&16&18&18\\
\hline\end{array}$$ where the length  $l(R)$ of a lattice $R$ is
the minimal number of generators of the discriminant group
$R^{\vee}/R$. Since $\rk(M_{d,n})=1+\rk(\mathbb{M}_n)$ and, if
$d\equiv 0\mod 2n$, $l(\langle 2d\rangle\oplus
\mathbb{M}_n)=1+l(\mathbb{M}_n)$, for every admissible $n$ and
$d\equiv 0\mod 2n$, $\rk(M_{d,n})\geq 2+l(M_{d,n})$, so by
\cite[Corollary 1.13.3]{NikulinIntQuadForms}, there is a unique
even hyperbolic lattice with the same rank, length, discriminant
group and form as $M_{d,n}$. Since $L_{d,n}'$ has all the
prescribed properties, we conclude that $L_{d,n}'\simeq M_{d,n}$.
Moreover, by \cite[Theorem 1.14.4]{NikulinIntQuadForms}, if $n< 7$
the lattice $L_{d,n}'\simeq M_{d,n}$ admits a unique, up to
isometry, primitive embedding in $\Lambda_{K3}$, and thus
determines a $(19-\rk(\Omega_n))$-dimensional family of K3
surfaces. If $n=7,8$, any primitive embedding of $L_{d,n}'\simeq
M_{d,n}$ in  the unimodular lattice $\Lambda_{K3}$, which exists
by results in \cite{GSprime, GSnonprime}, identifies the same
genus of the orthogonal complement $T_{d,n}$ of rank three and
signature $(2,1)$: we get respectively that
$A_{T_{d,7}}=\Z/7\Z\oplus \Z/2d\Z$ and $A_{T_{d,8}}=\Z/2\Z\oplus
Z/4\Z\oplus \Z/2d\Z$ with quadratic forms
$q_{T_{d,7}}=\left(-\frac{4}{7}\right)\oplus
\left(-\frac{1}{2d}\right)$ and
$q_{T_{d,8}}=\left(\frac{1}{2}\right)\oplus
\left(\frac{1}{4}\right)\oplus\left(-\frac{1}{2d}\right)$. It
follows from \cite[Proposition 1.15.1]{NikulinIntQuadForms} that
the primitive embedding of $L_{d,n}'\simeq M_{d,n}$ in
$\Lambda_{K3}$ is unique, up to isometry, if and only if $T_{d,n}$
is unique in its genus and the map $O(T_{d,n})\rightarrow
O(q_{T_{d,n}})$ is surjective. By \cite[Theorem
VIII.7.5]{MirMorr}, these two conditions hold in particular if the
discriminant quadratic form $q_{T_{d,n}}$ is $p$-regular for all
prime numbers $p\neq s$ and it is $s$-semiregular for a single
prime number $s$. The precise (and quite technical) definition of
$p$-regular and $p$-semiregular form can be found in
\cite[Definition VIII.7.4]{MirMorr}. An easy application of
\cite[Lemma VIII.7.6 and VIII.7.7]{MirMorr} implies that:
\begin{itemize}
 \item $q_{T_{d,7}}$ is $p$-regular if $p\neq 7$ and it is $7$-semiregular;
 \item $q_{T_{d,8}}$ is $p$-regular if $p\neq 2$ and it is $2$-semiregular.
\end{itemize}
Hence, also for $n=7,8$, $T_{d,n}$ is unique in its genus and the primitive embedding of $L_{d,n}'\simeq M_{d,n}$ in $\Lambda_{K3}$ is unique up to isometry,  and thus determines a $(19-\rk(\Omega_n))$-dimensional family of K3 surfaces.

Each K3 surface which is $ M_{d,n}$-polarized is contained in $\mathcal{L}_n\cap \mathcal{M}_n$ because there are primitive embeddings both of $\Omega_n$ and of $\mathbb{M}_n$ in its N\'eron--Severi group.
\endproof

\begin{proposition}\label{prop: non intersections}
Let $2\leq n\leq 8$ and $d\in\mathbb{N}$, $d\geq 1$. The lattice
$L_{d,n}$ is not isometric to any overlattice of finite index
(possibly 1) of $M_{e,n}$, for any $e$. In particular if $X$ is a
K3 surface such that $\NS(X)\simeq L_{d,n}$, then $X$ does not
admit a cyclic $n:1$ cover by a K3 surface and the families of the
$\left(L_{d,n}\right)$-polarized K3 surfaces are not (totally)
contained in $\mathcal{M}_n$.
\end{proposition}
\proof By Proposition \ref{prop: discriminant Omega and Mn},
$l(\Omega_{n})=2+l(\mathbb{M}_n)$. Hence $\l(L_{d,n})\geq
l(\Omega_n)=2+ l(\mathbb{M}_n)> l(M_{e,n})$. Since any overlattice
of $M_{e,n}$ has at most the length of $M_{e,n}$, the lattices
$L_{d,n}$ can not be isometric to any overlattice of the lattice
$M_{e,n}$.
\endproof
In conclusion we proved that there are components of
$\mathcal{L}_n$ (and of $\mathcal{M}_n$) which are completely
contained in $\mathcal{L}_n\cap\mathcal{M}_n$, but there are also
components of $\mathcal{L}_n$, which are not contained in
$\mathcal{M}_n$, and thus in $\mathcal{L}_n\cap\mathcal{M}_n$. It
is also true that there are components of $\mathcal{M}_n$ which
are not totally contained in $\mathcal{L}_n$ (see e.g. Theorem
\ref{theorem: intersection L2 and M2} for the case $n=2$.)

In the following proposition we construct an $n^2:1$ isogeny
between two K3 surfaces by using a third K3 surface, which is
$L_{d,n}'$-polarized, and we prove that generically this $n^2:1$
isogeny is not just the quotient by an automorphism group.

\begin{proposition}\label{prop: n^2 isogenies not quotients} Let $Z$ be a K3 surface, such that $\NS(Z)=L_{d,n}'$, let $X$ be the K3 surface which is a $n:1$ cyclic cover of $Z$ and $Y$ be the quotient of $Z$ by a symplectic automorphism of order $n$. Then there is an $n^2:1$ isogeny between $X$ and $Y$ but there is no finite group $G$ of automorphisms on $X$ such that $Y$ is birational to $X/G$.\end{proposition}
\proof By Theorem \ref{theorem: intersection Ln and Mn} the K3 surfaces $Z$ which are $L_{d,n}'$-polarized are $n:1$ isogenous to two K3 surfaces, $X$ and $Y$, respectively with the two $n:1$ isogenies $X\dra Z$ and $Z\dra Y.$ The composition of these two isogenies is an $n^2:1$ isogeny $X\dra Y$.

If there exists a group of automorphism $G$ as required, it has to be a group of symplectic automorphisms (otherwise the quotient $X/G$ would not be birational to a K3 surface). So $X$ should admit a group of symplectic automorphisms of order $n^2$. If $X$ admits a group $G$ of symplectic automorphisms, then $\left(\NS(X)^G\right)^{\perp}$ is a lattice (analogous to $\Omega_n$) which is unique in most of the cases. Its rank depends only on $G$ and it is known for every admissible $G$, see \cite{Hashimoto}. In particular, if $2\leq n\leq 8$, for every group $G_{n^2}$ of order $n^2$ acting symplectically on a K3 surface, the rank of the lattice is $\rk\left(\left(\NS(X)^{G_{n^2}}\right)^{\perp}\right)>\rk(\Omega_n)$. Hence, if a K3 surface $X$ admits $G_{n^2}$ as group of symplectic automorphisms, $\rho(X)>1+\Omega_n=\rho(Z)$. But $X$ and $Z$ are isogenous, hence $\rho(X)=\rho(Z)$ and thus $X$ can not admit a group of symplectic automorphisms of order $n^2.$\endproof

\begin{remark}\label{rem: Galois cover}{\rm In Proposition \ref{prop: n^2 isogenies not quotients} we proved that the $n^2:1$ isogeny $X\dashrightarrow Y$ is not
induced by a quotient map, so that the rational map
$X\dashrightarrow Y$ is not a Galois cover. Let us denote by $V$
its Galois closure, hence $V$ is a surface such that both $X$ and
$Y$ are birational to Galois quotients of $V$. Denoted by $G$ the Galois group
of the cover $V\dashrightarrow Y$ and by $H$ the subgroup of $G$
which is the Galois group of $V\dashrightarrow X$,
$H$ is not a normal subgroup of $G$, otherwise the rational map $X\dashrightarrow Y$ would be a Galois cover with Galois group $G/H$. The Kodaira dimension of the
surface $V$ is non negative (since $V$ covers K3 surfaces), but
moreover $V$ can not be a K3 surface. This can be proved applying
the same argument of the proof of Proposition \ref{prop: n^2 isogenies not quotients}: if $V$ were K3
surface, it would admit $G$ as group of symplectic automorphism,
but its Picard number is not big enough. We give more details on
the construction of $V$ and $G$ in the case $n=2$, see Section \ref{subsec: Galois closure 2^2:1 cover}}
\end{remark}

\section{Involutions}\label{sec:involutions}
In this section we restrict our attention to the case of the symplectic involutions (i.e. $n=2$). In this case several more precise and deep results are known about the relations between K3 surfaces admitting a symplectic involutions and K3 surfaces which are their quotients, hence we can improve the general results of the previous section and we can describe explicit examples. In particular we obtain: a complete description of the maximal dimensional components of the intersection $\mathcal{L}_2\cap \mathcal{M}_2$, in Theorem \ref{theorem: intersection L2 and M2}; infinite families of K3 surfaces such that for each K3 surface in a family there is another one in another family which is isogenous to it, in Corollary \ref{cor: infinite sets of families}; geometric examples in Sections \ref{sec: X1 and Y2} and \ref{subsec: U+N e U+E8}.

\subsection{Preliminary results on symplectic involutions and Nikulin surfaces}
For historical reasons, we refer to the K3 surfaces in $\mathcal{M}_2$ (i.e. the K3 surfaces which are cyclically $2:1$ covered by a K3 surface) as the Nikulin surfaces and to the lattice $\mathbb{M}_2$ as the Nikulin lattice, denoted by $N$($:=\mathbb{M}_2$). So we have
\begin{definition} A Nikulin surface $Y$ is a K3 surface which is the
minimal resolution of the quotient of a K3 surface $X$ by a
symplectic involution $\sigma$.
The minimal primitive sublattice of $\NS(Y)$ containing the curves arising from the desingularization of $X/\sigma$ is denoted by $N$ and it is called Nikulin lattice.
\end{definition}

If $\sigma$ is a symplectic involution on a K3 surface $X$, then
the fixed locus of $\sigma$ on $X$ consists of 8 isolated points.
The quotient surface $X/\sigma$ has 8 singularities of type
$A_1$, so the minimal resolution $X$ contains 8 disjoint
rational curves, which are the exceptional divisors over the
singular points of $X/\sigma$. Called $N_i$, $i=1,\ldots, 8$ their classes in the N\'eron--Severi group, the class
$\left(\sum_{i=1}^8N_i\right)/2$ is contained in the
N\'eron--Severi group. Indeed the union of the eight disjoint
rational curves is the branch locus of the double cover  $\widetilde{X}\ra Y$, where $\widetilde{X}$ is the blow up of $X$ in the 8 points fixed by $\sigma$. We will say that a set of rational curves is an even set if the sum of their classes divided by 2 is contained in the N\'eron--Severi group and that a set of nodes is an even set if
the curves resolving these nodes form an even set.

\begin{proposition}{\rm (\cite[Section 6]{NikSympl})}
The Nikulin lattice is an even negative definite lattice of rank 8
and its discriminant form is the same as the one of $U(2)^3$. It
contains 16 classes with self-intersection $-2$, i.e. $\pm N_i$,
$i=1,\ldots, 8$. A $\Z$-basis of $N$ is given by
$\left(\sum_{i=1}^8N_i\right)/2$, $N_i$, $i=1,\ldots
7$.\end{proposition}

As in the previous section, we will denote by $M_{e,2}$ the lattice $\langle 2e\rangle\oplus N$.
\begin{proposition}\label{prop: NS(X) Nikulin surfaces}$(a)$ A K3 surface $Y$ is a Nikulin surface if and only if the lattice
$N$ is primitively embedded in $\NS(Y)$.

$(b)$ The minimal Picard number of a Nikulin surface is 9.

$(c)$ There exists an even overlattice of index two
of $M_{e,2}$ in which $N$ is primitively
embedded if and only if $e$ is even. In this case, this lattice is unique up to isometry and denoted by $M_{e,2}'$.

$(d)$ If $Y$ is a Nikulin surface with Picard number 9, then $\NS(Y)$ is
isometric either to $M_{e,n}$ or to $M_{e,n}'$ for a certain $e$.
\end{proposition}
\proof The point $(a)$ is  Theorem \ref{theorem: Y has cyclic cover iff Mn is in  NS(Y)}, the point $(b)$ is proved in Corollary \ref{corollary: NS of K3 in Ln and Mn} in the case $n=2$. The point $(c)$ is proved in \cite[Proposition 2.2]{GS} and the point $(d)$ in \cite[Proposition 2.1]{GS}.\endproof

In the case $n=2$, the lattice $\Omega_2$ is known to be isometric to $E_8(-2)$. As in the previous section, we will denote by $L_{d,2}$ the lattice $\langle 2d\rangle\oplus
E_8(-2)$ and by $L_{d,2}'$ the overlattice of index two of $L_{d,2}$
such that $L_{d,2}'$ is even and $E_8(-2)$ is primitively embedded
in $L_{d,2}'$.
\begin{proposition}\label{prop: NS(Y) con inv sympl} $(a)$ A K3 surface $X$ admits a symplectic involution $\sigma$ if and only if the lattice
$E_8(-2)$ is primitively embedded in $\NS(X)$.

$(b)$ If $X$ admits a symplectic involution, $\rho(X)\geq 9$.

$(c)$ There exists an even overlattice of index two
of $L_{d,2}$ in which $E_8(-2)$ is primitively
embedded if and only if $d$ is even. In this case, this lattice is unique up to isometry and is $L_{d,2}'$.

$(d)$ If $X$ is a K3 surface admitting a symplectic involution and with Picard number 9, then $\NS(X)$ is isometric either to $L_{d,2}$ or to $L_{d,2}'$ for a certain $d$.

\end{proposition}
\proof The Proposition follows directly by \cite[Propositions 2.2 and 2.3]{vGS} (and the points $(a)$, $(b)$ and $(c)$ were already proved in the more general setting of automorphisms of order $n$ in the previous Section).\endproof

The main result which is known for involutions and is not yet stated in the more general case of symplectic automorphisms of order $n$, is the explicit relation between the N\'eron--Severi group of a K3 surface which admits a symplectic involution and the N\'eron--Severi group of the K3 surface which is its quotient.
\begin{proposition}\label{prop: NS(Y) determines NS(X)}{\rm (\cite[Corollary 2.2]{GS})}
Let $X$ be a K3 surface with a symplectic involution $\sigma$ and
$Y$ be the minimal resolution of $X/\sigma$. Then:
\begin{itemize} \item $\NS(X)\simeq L_{e,2}$ if and only if $\NS(Y)\simeq
M_{2e,2}'$ \item $\NS(X)\simeq L_{2e,2}'$ if and only if $\NS(Y)\simeq
M_{e,2}$.\end{itemize}
\end{proposition}

\subsection{The intersection $\mathcal{L}_2\cap\mathcal{M}_2$ and infinite towers of isogenous K3 surfaces}

Since we know the structure of all the possible N\'eron--Severi groups of Nikulin surfaces of minimal Picard number (by Proposition \ref{prop: NS(X) Nikulin surfaces}) and all the possible N\'eron--Severi groups of K3 surfaces of minimal Picard number admitting a symplectic involution (by Proposition \ref{prop: NS(Y) con inv sympl}), we are able to give the following refinement of the Theorem \ref{theorem: intersection Ln and Mn} and of the Proposition \ref{prop: non intersections}.
\begin{theorem}\label{theorem: intersection L2 and M2}
A Nikulin surface $Y$ such that $\rho(Y)=9$ admits a symplectic involution if and only if $\NS(Y)\simeq M_{2d,2}(\simeq L_{2d,2}')$.

A K3 surface $X$ admitting a symplectic involution such that $\rho(X)=9$ is a Nikulin surface if and only if $\NS(X)\simeq L_{2d,2}'(\simeq M_{2d,2})$.

So $$\mathcal{L}_2\cap\mathcal{M}_2\supset\bigcup_{d\in \N_{>0},\ d\equiv0 (2)}\mathcal{P}(M_{2d,2}).$$
\end{theorem}
\proof By Theorem \ref{theorem: intersection Ln and Mn} $M_{2d,2}\simeq L_{2d,2}'$ and thus if $\NS(Y)\simeq M_{2d,2}$, then $Y$ admits a symplectic involution. Similarly if $\NS(X)\simeq L_{2d,2}'$, $X$ is a Nikulin surface. It remains to prove that if a K3 surface is in $\mathcal{L}_2\cap\mathcal{M}_2$, and its Picard number is 9, then its N\'eron--Severi can not be isometric to $M_{e,2}'$, to $M_{e,2}$ for an odd $e$ or to $L_{f,2}$ with $f\in\mathbb{N}_{>0}$.
The argument is similar to the one of Proposition \ref{prop: non intersections}.

By Proposition \ref{prop: NS(Y) con inv sympl}, if $Y$ is a
Nikulin surface, its N\'eron--Severi group is either isometric to
$M_{e,2}$ or to $M_{2e,2}'$. By Proposition \ref{prop: NS(X)
Nikulin surfaces} if $X$ is a K3 surface admitting a symplectic
involution, its N\'eron--Severi group is either isometric to
$L_{d,2}$ or to $L_{2d,2}'$. So if a K3 surface has both
properties (i.e. it is in $\mathcal{L}_2\cap \mathcal{M}_2$ and
has Picard number 9), its N\'eron--Severi group is isometric both
to a lattice in $\{M_{e,2}, M_{2e,2}'\}$ and to a lattice in
$\{L_{d,2}, L_{2d,2}'\}$. Hence we are looking for pairs of
lattices, one in $\{M_{e,2}, M_{2e,2}'\}$  and one in $\{L_{d,2},
L_{2d,2}'\}$, which are isometric. If two lattices are isometric,
they have the same length. We observe that  $l(M_{e,2})=1+l(N)=7$,
$l(M_{2e,2}')=1+l(N)-2=5$, $l(L_{d,2})=1+l(\Omega_2)=9$,
$l(L_{2d,2}')=1+l(\Omega_2)-2=7$. In particular, the unique
possible pair of lattices as required is given by $(M_{e,2},
L_{2d,2}')$. Since if two lattices are isometric they have the
same discriminant, one obtains that $e=2d$.\endproof

\begin{corollary}\label{cor: related Nikulin surfaces}
Two Nikulin surfaces $Y$ and $\hat{Y}$ with Picard number 9 are isogenous by a chain of quotients by involutions if and only if one of the following equivalent conditions hold:\\
$(i)$ $\NS(Y)\simeq
M_{d,2}$, $\NS(\hat{Y})=M_{e,2}$, and there exists $m\in \N_{>0}$ such
that either $d=2^me$ or $e=2^md$;\\
$(ii)$ $T_Y\simeq U\oplus U\oplus N\oplus \langle -2d\rangle$, $T_{\hat{Y}}\simeq U\oplus U\oplus N\oplus \langle -2e\rangle$ and there exists $m\in \N_{>0}$ such
that either $d=2^me$ or $e=2^md$.
\end{corollary}
\proof We can assume that $\hat{Y}$ is obtained by iterated  quotients from $Y$. Then $Y$ admits a symplectic involution $\sigma$ and, by Theorem \ref{theorem: intersection L2 and M2}, there exists an even $d$ such that $\NS(Y)\simeq M_{d,2}\simeq L_{d,2}'$. So $Y$ is the cover of a K3 surface $Z$ with N\'eron--Severi group $M_{d/2,2}$ (by Proposition \ref{prop: NS(Y) determines NS(X)}). If $d$ is odd, then the process stops and $\hat{Y}$ is necessarily $Z$; otherwise, $\NS(Z)\simeq M_{d,2}\simeq L_{d,2}'$ and $Z$ is the cover of a K3 surface $Z$ with N\'eron--Severi group $M_{d/4,2}$. Iterating, if possible, this process $m$ times, one obtains Nikulin surfaces with N\'eron--Severi group isometric to $M_{d/2^{m},2}$. In particular, one never obtains lattices isometric to $M_{e,2}'$ (for a certain $e$) as N\'eron--Severi groups of a Nikulin surface obtained by iterated quotients from $Y$.

Vice versa, if $\NS(\hat{Y})\simeq M_{e,2}$ for a certain $e$, $\hat{Y}$ is covered by a K3 surface $W$ with $\NS(W)\simeq L_{2e,e}'\simeq M_{2e,2}$ (by Proposition \ref{prop: NS(Y) determines NS(X)}). So $W$ is a Nikulin surface, $2:1$ covered by a K3 surface with N\'eron--Severi group isometric to $L_{4e,e}'\simeq M_{4e,2}$. Reiterating this process $m$ times one obtains that $\hat{Y}$ is isogenous to a Nikulin surface whose N\'eron--Severi lattice is isometric to $M_{h,e}$ with  $h=2^md$.

The equivalent statement for the transcendental lattice follows by the fact that if the N\'eron--Severi group of a K3 surface is isometric to $M_{d,2}$, then its transcendental lattice is isometric to $U\oplus U\oplus N\oplus\langle -2d\rangle$ (since the discriminant form of the latter is minus the discriminant form of $M_{d,2}$, and in this case the transcendental lattice is uniquely determined by its genus).\endproof

We determined an infinite number of infinite series of Nikulin
surfaces of Picard number 9 related by iterated quotients by symplectic involutions. More precisely we proved the
following.
\begin{corollary}\label{cor: infinite sets of families} For every $d\in\mathbb{N}$, if $\NS(Y)\simeq M_{d,2}$ there exists an infinite number of K3 surfaces $Y_m$ isogenous to $Y$. In particular for each $m$ there exists at least one K3 surface $Y_m$ with an isogeny of
degree $2^m$ to $Y$ whose N\'eron--Severi group is isometric to
$\NS(Y_m)=M_{2^md,2}$. The transcendental lattice of $Y$ is
$T_Y\simeq U\oplus U\oplus N\oplus \langle -2d\rangle$ and for
each $m$ the one of $Y_m$ is $T_{Y_m}\simeq U\oplus U\oplus
N\oplus \langle -2^{m+1}d\rangle$.
\end{corollary}

\begin{remark}{\rm The $M_{2^{m+2}d,2}$-polarized K3 surfaces can be interpreted as moduli spaces of twisted sheaves on $M_{2^{m}d,2}$-polarized K3 surfaces. Let $e_1,e_2$ be a standard basis of the first copy of $U$ inside the K3 lattice $\Lambda_{K3}$ and choose a  primitive embedding of $M_{2^{m+2}d,2}$ in $\Lambda_{K3}$ so that a generator of $\langle 2^{m+1}d\rangle$ is $e_1+2^mde_2$ and $N$ is embedded in $U^\perp$.
Given $S\in\mathcal{P}(M_{2^{m}d,2})$ generic, the transcendental lattice is
$T_S=U^{\oplus 2}\oplus N\oplus\langle -2^{m+1}d\rangle$, and with
our previous choice it is easy to see that $\langle
-2^{m+1}d\rangle$ is generated by $t:=e_1-2^mde_2$. The $B$-field
$B=\frac{e_2}{2}\in H^2(S,\Q)$ is a lift for the Brauer class
$\beta:T_S\ra \Z/2\Z$ given by $v\mapsto (v,2B)$.
It is easy to see
that $T(S,B)\cong\ker\beta=U^{\oplus 2}\oplus N\oplus \langle
-2^{m+3}d\rangle\subset H^*(S,\Z)$, where the last summand is spanned by $(0,2t,1)$.
Moreover, the orthogonal of $T(S,B)$ inside the Mukai lattice
$H^*(S,\Z)$ is the generalized Picard group $\Pic(S,B)$, which is
the sublattice spanned by $f_1:=(0,0,1)$, $f_2:=(2,e_2,0)$,
$f_3:=(0,e_1+2^mde_2,0)$ and $(0,b_i,0)$ with $b_1,\cdots,b_8\in
M_{2^{m+2}d,2}$ a basis of the lattice $N$, and has quadratic form
 \[
 \left(\begin{array}{ccc}
        0&2&0\\
        2&0&1\\
        0&1&2^{m+1}d
       \end{array}\right)\oplus N.
 \]
 The isotropic element $v:=2^mdf_2-f_3$ now satisfies $((\Z v)^\perp\cap\Pic(S,B))/\Z v\simeq M_{2^{m+2}d,2}$. Hence, the moduli space of stable twisted sheaves $M_v(S,\beta)$ is a smooth $M_{2^{m+2}d,2}$-polarized K3 surface. It is an interesting open question to see whether the isogeny of degree $4$ which we constructed here coincides with the one induced by a twisted universal family on $S\times M_v(S,\beta)$ or not (for further details see \cite[Theorem 0.1]{Huybrechts}).
}
\end{remark}

\subsection{The Galois closure of $2^2:1$ covers}\label{subsec: Galois closure 2^2:1 cover}
Let $X_d$ be a K3 surface such that $\NS(X_d)=M_{d,2}$. Then
$X_{2d}\in\mathcal{L}_2\cap \mathcal{M}_2$ and there are the two
Galois covers $X_{4d}\dashrightarrow X_{2d}$ and
$X_{2d}\dashrightarrow X_d$. The composition of these two maps is
a $2^2:1$ isogeny, not induced by a Galois cover, by Proposition
\ref{prop: n^2 isogenies not quotients}. As observed in Remark
\ref{rem: Galois cover} there exist a surface $V$, a group
$G\subset\Aut(V)$ and a subgroup $H$ of $G$ such that $V/G$ is birational to $X_d$, and $V/H$ is birational to $X_{4d}$. Here we
construct the surface $V$ and the group $G$, proving the following
\begin{proposition} The group $G$ is the dihedral group of order 8 and $V$ is a $(\Z/2\Z)^2$ Galois cover of $X_{2d}$, whose branch locus $B$ is the union of 16 smooth rational curves. If $B$ is normal crossing, then $V$ is a positive Kodaira dimension smooth surface such that $h^{1,0}(V)=0$ and $h^{2,0}(V)\geq 35$.
\end{proposition}
To prove the Proposition one constructs a $\left(\Z/2\Z\right)^2$ Galois cover of $X_{2d}$ (see Section \ref{subsub: bidouble cover X2d}) by a surface denoted by $V$. Then one constructs a $\left(\Z/2\Z\right)^2$ Galois cover of $X_d$ (see Section \ref{subsub: bidouble cover Xd}), and eventually one proves that these two covers can be pasted to obtain a unique Galois cover by the dihedral group of order 8 (see Section \ref{subsub:  Galois cover Xd}). In order to obtain the $\left(\Z/2\Z\right)^2$ covers one compares the branch loci of the $2:1$ maps $X_{4d}\dashrightarrow X_{2d}$ and $X_{2d}\dashrightarrow X_d$.

Here we do not consider the Galois closure of $2^n:1$ isogenies given in Corollary \ref{cor: infinite sets of families}, but for any fixed $n$ a priori one can iterate the previous process.

In the following we will call the $(\Z/2\Z)^2$ Galois covers bidouble covers as in \cite{Catanese}, where all the basic definitions and properties of these covers can be found.

\subsubsection{A bidouble cover of the surface $X_{2d}$}\label{subsub: bidouble cover X2d}

Let us denote by $N_1,\ldots ,N_8\subset X_{2d}$ the rational
curves which are the branch locus of the double cover
$X_{4d}\dashrightarrow X_{2d}$. Let us denote by $\sigma_{2d}$ the
symplectic involution on $X_{2d}$ such that $X_d$ is birational to
$X_{2d}/\sigma_{2d}$. The curves $N_i$, $i=1,\ldots, 8$ are not
preserved by $\sigma_{2d}$, since $ \langle
H\rangle:=\NS(X_{2d})^{\sigma_{2d}^*}$ is positive definite and
more precisely $H^2=4d$. Set $N_i':=\sigma_{2d}(N_i)$. Hence we
found two even sets of eight rational curves on $X_{2d}$:
$\left\{N_1,\ldots, N_8\right\}$ and $\left\{N_1',\ldots,
N_8'\right\}$. Let $D_1:=\sum_{i=1}^8N_i$, $D_2:=\sum_{i=1}^8
N_i'$, $D_3:=0$ and $2L_i:=D_j+D_k$, $i,j,k\in\{1,2,3\}$. The six
divisors $D_j$, $L_i$, $j,i=1,2,3$ in $\NS(X_{2d})$ satisfy the
conditions which define a bidouble cover, so there exists a
surface $V$ such that $(\Z/2\Z)^2\in\Aut(V)$ and $V/(\Z/2\Z)^2$ is
(birational to) $X_{2d}$ (see \cite[Section 2]{Catanese}).
Moreover, there are three surfaces which are double covers of
$X_{2d}$ branched respectively along the curves supported on
$2L_1$, $2L_2$ and $2L_3$; all of them are $2:1$ covered by $V$.
Since $L_2=\sum_{i=1}^8N_i/2$, the double cover of $X_{2d}$
branched on $2L_2$ is a non-minimal model of $X_{4d}$. We denote
the cover branched on the curves in the support of $2L_2$ by
$\widetilde{X_{4d}}$. Similarly the double cover of $X_{2d}$
branched on $\cup_i N_i'$ is the blow up of a K3 surface,
$X_{4d}'$, in 8 points and it will be denoted by
$\widetilde{X_{4d}'}$. The N\'eron--Severi group of the K3 surface
$X_{4d}'$ is determined by the one of $X_{2d}$, by Proposition
\ref{prop: NS(Y) determines NS(X)}, and thus it is isometric to
$M_{4d}$. We obtain the following diagram:

\begin{align}\label{eq: bidouble X2d}\xymatrix{&V\ar[dr]\ar[d]\ar[dl]&\\
W\ar[dr]&\widetilde{X_{4d}}\ar[d]^{\widetilde{\pi_{4d}}}&\widetilde{X_{4d}'}\ar[dl]^{\widetilde{\pi'_{4d}}}\\&X_{2d}&}\end{align}
The surfaces $W$ and $V$ have non negative Kodaira dimension,
because they are covers of K3 surfaces.

Let us now suppose that the intersections $N_i\cap N_j'$ are transversal, and thus both the branch divisors of $W\ra X_{2d}$ and of $V\ra X_{2d}$ are normal crossing. Under this assumption, $V$ is smooth and the birational invariants of $W$ and $V$ depend only on $L_j^2$, for $j=1,2,3$.
The surface $W$ is the double cover of $X_{2d}$ branched on the
reducible curve which is the support of $2L_3$, i.e. on the curve
$\bigcup_{i=1}^8(N_i\cup N_i')$. Since $N_i+N_i'$ is an effective
($\sigma_{2d}^*$)-invariant divisor and $H$ is the ample generator of
$\NS(X_{2d})^{\sigma_{2d}^*}$, there exists a positive integer
$k_i$ such that $N_i+N_i'= k_i H$. Then
$L_3^2=\left.\left(\sum_{i=1}^8 k_i
H\right)^2\right/4=d\left(\sum_{i=1}^{8}k_i\right)^2$
and
$$\chi(W)=4+\frac{d}{2}\left(\sum_{i=1}^{8}k_i\right)^2,\
h^{2,0}(W)=3+\frac{d}{2}\left(\sum_{i=1}^{8}k_i\right)^2 \mbox{
and } h^{1,0}(W)=0.$$
The singularities of $W$ are in the inverse image of the singular
points of $\bigcup_{i=1}^8(N_i\cup N_i')$ and $V$ is a double
cover of $W$ branched on its singular points. The invariants of $V$ can be computed by \cite[Section 2]{Catanese}, from which one
obtains \[
h^{2,0}(V)=h^{2,0}(W)=3+\frac{d}{2}\left(\sum_{i=1}^{8}k_i\right)^2\geq
3+32d\geq 35,\ \ h^{1,0}(V)=h^{1,0}(W)=0.
\]
Hence $V$ is a surface with non negative Kodaira dimension, $h^{2,0}(V)\geq 35$ and $h^{1,0}(V)=0$, so its
Kodaira dimension is necessarily positive.

\subsubsection{A bidouble cover of the surface $X_{d}$}\label{subsub: bidouble cover Xd}

The surface $X_d$ is the desingularization of the quotient of
$X_{2d}/\sigma_{2d}$ and we will denote by $R_1,\ldots, R_8$ the
eight disjoint rational curves resolving the singularities of
$X_{2d}/\sigma_{2d}$. Equivalently, the double cover of $X_{d}$
branched on  $\bigcup_i R_i$ is birational to $X_{2d}$.
Denoted by $\pi_{2d}:X_{2d}\ra X_{2d}/\sigma_{2d}$ the quotient
map, one observes that $\pi_{2d}(N_i)=\pi_{2d}(N_i')$,
$i=1,\ldots, 8$, and $\pi_{2d}(N_i)$ is a rational curve singular
in the points $\pi_{2d}(N_i\cap N_i')$. We denote by
$\overline{N_i}$ the strict transform on $X_d$ of the curve
$\pi_{2d}(N_i)$. The curves $\overline{N_i}$ could be singular and
the set $\{\overline{N_1},\ldots, \overline{N_8}\}$ is a divisible
set. Moreover, since $N_i+N_i'=k_iH\subset
\NS(X_{2d})^{\sigma_{2d}^*}$, one has $(\pi_{2d})_*(N_i+N_i')\subset
\NS(X_{2d}/\sigma_{2d})$. Hence $\overline{N_i}\subset
(N^{\perp_{\NS(X_d)}})$.

The sets $\{R_1,\ldots, R_8\}$ and $\{\overline{N_1},\ldots
\overline{N_8}\}$ are two 2-divisible sets of curves, which allow
us to construct a bidouble cover of $X_d$, whose data are
$\Delta_1:=\sum_{i=1}^8 R_i$,
$\Delta_2:=\sum_{i=1}^8\overline{N_i}$,
$\Delta_3:=\sum_{i=1}^8\left(R_i+\overline{N}_i\right)$,
$2\Gamma_i:=\Delta_j+\Delta_k$, with $\{i,j,k\}=\{1,2,3\}$. The
double cover $\widetilde{X_{2d}}\ra X_{d}$ is branched over
$\cup_iR_i$, i.e. the curve in the support of $2\Gamma_2$. It
induces a double cover of $\widetilde{X_{2d}}$ branched over
$\cup_i\left(\widetilde{N_i}+\widetilde{N_i'}\right)$, where
$\widetilde{N_i}$ (resp. $\widetilde{N_i'}$) is the strict
transform on $\widetilde{X_{2d}}$ of the curve $N_i$ (resp.
$N_i'$). Let us denote by $\widetilde{W}$ the surface double cover
of $\widetilde{X_{2d}}$ branched on
$\cup_i\left(\widetilde{N_i}+\widetilde{N_i'}\right)$. So we have
the following diagram:
\begin{align}\label{eq: bidouble Xd}\xymatrix{&\widetilde{W}\ar[dr]\ar[d]\ar[dl]&\\
B\ar[dr]&A\ar[d]&\widetilde{X_{2d}}\ar[dl]^{\widetilde{\pi_{2d}}}\\&X_{d}&}\end{align}
where $\widetilde{\pi_{2d}}:\widetilde{X_{2d}}\ra X_d$ is induced by $\pi_{2d}$.
\subsubsection{The $\mathcal{D}_4$ cover of $X_{d}$}\label{subsub: Galois cover Xd}
Both the diagrams \eqref{eq: bidouble  Xd} and \eqref{eq: bidouble  X2d} induce a $2:1$ rational map $W\dashrightarrow X_{2d}$, which is (birationally) the double of cover of $X_{2d}$ branched on $\bigcup_{i=1}^8 \left(N_i\cup N_i'\right)$. Hence these diagrams can be pasted to obtain the following, where all the arrows are rational maps of generically degree 2

\begin{align}\label{eq: D4 cover}\xymatrix{&&V\ar@{-->}[dr]\ar@{-->}[d]\ar@{-->}[dl]&\\
&W\ar@{-->}[dr]\ar@{-->}[d]\ar@{-->}[dl]&X_{4d}\ar@{-->}[d]^{\pi_{4d}}&X_{4d}'\ar@{-->}[dl]^{\pi'_{4d}}\\B\ar@{-->}[dr]&A\ar@{-->}[d]&X_{2d}\ar@{-->}[dl]^{\pi_{2d}}&\\&X_{d}
}\end{align} We already proved that the $4:1$ covers
$X_{4d}\dashrightarrow X_{d}$ and $X_{4d}'\dashrightarrow X_d$ are
not Galois covers in Proposition \ref{prop: n^2 isogenies not
quotients}. On the other hand, the cover $4:1$ $W\ra X_d$ is a
Galois cover (indeed a bidouble cover), by construction. Since
$V\dashrightarrow X_{2d}$ is constructed as bidouble cover, the
cover involution of the cover $W\dashrightarrow X_{2d}$, lifts to
an involution of $V$ (which is the cover involution of
$V\dashrightarrow X_{4d}$). Hence one obtains that the cover
$V\dashrightarrow X_{d}$ is a Galois $8:1$ cover. The cover group
$G$ is an order 8 group, which admits non normal subgroups of
order 2 (otherwise $X_{4d}\dashrightarrow X_d$ should be a Galois
cover). Hence $G\simeq \mathcal{D}_4$, the dihedral group of order
8. We recall that $\mathcal{D}_4:=<s,r|s^2=1, r^4=1, rs=sr^{-1}>$.
The center $H$ of $G$ is $H:=<r^2>$ and the quotient of $V$ by $H$
is birational to $W$. So we conclude that the Galois cover is
given by the surface $V$ on which acts the group $G=\mathcal{D}_4$.

\subsection{The K3 surface $X_2\in\mathcal{L}_2\cap\mathcal{M}_2$ }\label{sec: X1 and Y2}

By Proposition \ref{theorem: intersection L2 and M2}, for every
even $d$, if a K3 surface $X_d$ is such that $\NS(X_d)\simeq
L_{d,2}'$, then $X_d$ admits a symplectic involution and it is
$2:1$ cyclically covered by a K3 surface. Here we describe
geometrically these properties for the minimum possible value of
$d$, i.e. for $d=2$: let $X_2$ be a K3 surface with $\NS(X_2)\simeq L_{2,2}'$. It admits
an involution $\sigma$ and by Proposition \ref{prop: NS(Y)
determines NS(X)} the K3 surface $Y_1$ which is the
desingularization of $X_2/\sigma$ is such that
$\NS(Y_1)=M_{1,2}\simeq \langle 2\rangle \oplus N$. Since $\NS(X_2)\simeq M_{2,2}\simeq\langle 4\rangle\oplus N$ (by
Proposition \ref{theorem: intersection L2 and M2}), the surface
$X_2$ is $2:1$ covered by a K3 surface $X_4$, whose
N\'eron--Severi group is $\NS(X_4)\simeq L_{4,2}'$ (by Proposition
\ref{prop: NS(Y) determines NS(X)}). Since $X_2$ is a Nikulin
surface, there are 8 disjoint rational curves, which resolve the
singularities of the quotient of $X_4$ by a symplectic involution. Thus the surface $X_2$ admits two different descriptions
according to the interpretation of it as K3 surface with a
symplectic involution or as Nikulin surface.
These descriptions are associated to different projective models,
induced by different (pseudo)ample divisors. Here we recall these
descriptions and we explain how to pass from one to the other.\\

By \cite[Section 3.5]{vGS}, any K3 surface $X_2$ such that $\NS(X_2)\simeq L_{2,2}'$ is described as bidouble cover of $\mathbb{P}^2$ as follows: one considers two smooth plane curves $B$ and $C_0$ of degree respectively 4 and 2 in $\mathbb{P}^2$. The double cover of $\mathbb{P}^2$ branched on $B\cup C_0$ is a surface singular in eight points, the inverse image of $B\cap C_0$. The resolution of this surface is the K3 surface $X_1$ such that $\NS(X_1)\simeq M_{1,2}$ and the eight rational curves arising from this resolution will be denoted by $R_i$, $i=1,\ldots, 8$. The curves $R_1,\ldots, R_8$ form an even set of rational curves on $X_1 $ and the double cover of $X_1$ branched on $\cup_i R_i$ is, by construction, a K3 surface $X_2$ such that $\NS(X_2)\simeq L_{2,2}'$. The choice of the curves $B\cup C_0$ totally determines the surfaces $X_1$ and $X_2$. To construct the bidouble cover one considers also the double cover of $\mathbb{P}^2$ branched on $C_0$ and the double cover of $\mathbb{P}^2$ branched on $B$. The first surface is a quadric $Q\simeq \mathbb{P}^1\times\mathbb{P}^1\subset\mathbb{P}^3$, the latter a del Pezzo surface of degree 2, denoted in the following by $dP$. Hence one has the following diagram, where all the arrows are rational maps of degree 2:
\begin{align}\label{eq: bidouble P2}\xymatrix{&X_2\ar[dr]^{\pi_2}\ar[d]^{\pi_{dP}}\ar[dl]_{\pi_{Q}}&\\
\mathbb{P}^1\times\mathbb{P}^1\simeq Q\ar[dr]_{q_1}&dP\ar[d]^{q_2}&X_{1}\ar[dl]^{q_3}\\&\mathbb{P}^2&}\end{align}
The N\'eron--Severi group of $X_2$ is isometric to $L_{2,2}'$, hence it is an overlattice of index two of $\langle 4\rangle\oplus E_8(-2)$. The linear system of the ample divisor
$L$, orthogonal to $E_8(-2)$ in $L_{2,2}'$ exhibits $X_2$ as
double cover of a quadric $\mathbb{P}^1\times\mathbb{P}^1$ in $\mathbb{P}^3$. One can assume that
the class generating $L_{2,2}'/L_{2,2}$ is $E_1:=(L+e_1)/2$, where
$e_i$ is a standard basis of $E_8(-2)$ (i.e. $e_ie_{i+1}=2$ if
$i=1,\ldots 6$, $e_3e_8=2$, $(e_i)^2=-4$ and the other
intersections are 0). Then the divisor $E_1$ is a nef divisor and
the map associated to its linear system
$\varphi_{|E_1|}:X_2\ra\mathbb{P}^1$ is a genus 1 fibration. The
action of $\sigma^*$ on $\NS(X_2)$ is the identity on the subspace
$\langle L\rangle$ and minus the identity on the subspace
$L^{\perp}\simeq E_8(-2)$. So the image of $E_1$ by $\sigma^*$ is
the nef divisor $E_2:=(L-e_1)/2=E_1-e_1$ (see \cite[Section 3.5]{vGS}). The two maps $\varphi_{|E_i|}$,
$i=1,2$ are the maps on the rulings of the quadric $Q\subset
\mathbb{P}^3$ image of the map
$\varphi_{|L|}=\varphi_{|E_1+E_2|}$. In particular the set of divisors $\{E_1, e_1,\ldots e_8\}$ is a basis of $\NS(X_2)$.

By \eqref{eq: bidouble  P2}, it follows that $X_2$ admits three commuting involutions, the covering involutions of the three double covers $\pi_Q$, $\pi_{dP}$, $\pi_2$. The latter involution is the symplectic involution $\sigma$, the others will be denoted by $\iota_Q$ and $\iota_{dP}$ respectively.
\begin{proposition}\label{prop: three involutions on NS(X2)}
The involutions $\iota_Q$ and $\iota_{dP}$ are non-symplectic involutions and their composition is the symplectic involution $\sigma$. The group $\langle \iota_Q,\iota_{dP}\rangle$ is isomorphic to $\left(\Z/2\Z\right)^2$ and it is the Galois group of the $2^2:1$ cover $\pi:X_2\dashrightarrow\mathbb{P}^2$.

The induced three involutions on $\NS(X_{2})$ act as follows on the
basis $\{E_1,e_1,\ldots e_8\}$:
\begin{eqnarray*}
\begin{array}{llllll}
\sigma^*(E_1)=E_1-e_1,&\sigma^*(e_i)=-e_i,&i=1,\ldots,8
\\
\iota_{Q}^*(E_1)=E_1,&
\iota_{Q}^*(e_1)=e_1,&
\iota_{Q}^*(e_2)=-e_1-e_2,&
\iota_{Q}^*(e_j)=-e_j,\\
\iota_{dP}^*(E_1)=E_1-e_1,&\ \iota_{dP}^*(e_1)=-e_1,& \iota_{dP}^*(e_2)=e_1+e_2&\ \iota_{dP}^*(e_j)=e_j,
\end{array}
\end{eqnarray*}
where $j=3,\ldots, 8$.
\end{proposition}
\proof The action of $\sigma$ is minus the identity on $\left(\NS(X_2)^\sigma\right)^{\perp}\simeq E_8(-2)\subset \NS(X_2)$ and we chose the basis of $\NS(X_2)$ in such a way that the divisors $e_i$, $i=1,\ldots, 8$ span exactly $\left(\NS(X_2)^\sigma\right)^{\perp}\simeq E_8(-2)$. Moreover we chose $L$ to be the orthogonal to $\langle e_i\rangle_{i=1,\ldots 8}$ and thus $\sigma^*(L)=L$. By the definition of $E_1(=\left(L+e_1\right)/2)$ one obtains $\sigma^*(E_1)=(L-e_1)/2=E_1-e_1$.

The automorphism $\iota_Q$ is such that $X_2/\iota_Q$ is a rational surface and thus $\iota_Q$ is non symplectic and $X_2/\iota_Q$ is smooth. Since $X_2/\iota_Q$ is $\mathbb{P}^1\times\mathbb{P}^1$, $\rk(\NS(X_2)^{\iota_Q})=2$ and $\NS(X_2)^{\iota_Q}$ is generated by the divisors which induce the maps $X_2\ra\mathbb{P}^1$ given by the composition of the quotient map $\pi_Q:X_2\ra\mathbb{P}^1\times\mathbb{P}^1$ with the projection on the first, respectively second, factor. These maps are $\varphi_{|E_1|}: X_2\ra\mathbb{P}^1$ and $\varphi_{|E_2|}: X_2\ra\mathbb{P}^1$. So $\NS(X_2)^{\iota_Q}=\langle E_1, E_2\rangle$ and $\iota_Q$ acts as minus the identity on $\left(\NS(X_2)^{\iota_Q}\right)^\perp$. So $\iota_Q^*(e_j)=-e_j$ if $j=3,\ldots, 8$, $\iota_Q^*(E_1)=E_1$, and $\iota_Q^*(E_2)=\iota_Q^*(E_1-e_1)=E_1-\iota_Q^*(e_1)=E_1-e_1=E_2$. It follows $\iota_Q^*(e_1)=e_1$. In order to find the image of $e_2$ it suffices to recall that $\iota_Q^*$ is an involution and  that $\left(\iota_Q^*(e_2)\right)\cdot \left(\iota_Q^*(D)\right)=e_2D$ for any divisor $D\in \NS(X_2)$.
The group $\langle\iota_Q,\sigma\rangle$ is by construction the Galois group of the cover $\pi:X_2\ra \mathbb{P}^2$, so it is isomorphic to $\left(\Z/2\Z\right)^2$ and contains three different involutions, each of them is the composition of the other two. In particular $\iota_{dP}=\iota_Q\circ\sigma$ and so $\iota_{dP}^*=\sigma^*\circ\iota_Q^*$ and $\iota_{dP}$ is non-symplectic.\endproof

We already observed that the classes $E_1:=(L+e_1)/2$ and $E_2:=(L-e_1)/2$ induce two elliptic fibrations $\varphi_{|E_i|}:X_2\ra\mathbb{P}^1$. By the properties of these elliptic fibrations we will be able to identify the classes of irreducible rational curves on $X_2$ and in particular 8 classes which span the Nikulin lattice.
The following proposition gives the explicit
isometry between $L_{2,2}'$ and $M_{2,2}$ and shows directly that
the surface $X_2$ admits a 2:1 rational double cover by another K3
surface, thus it provides an explicit geometric interpretation of
Theorem \ref{theorem: intersection Ln and Mn} in the case $n=2$.

\begin{proposition}\label{prop: L4 simeq M4'}
Both the genus 1 fibrations $\varphi_{|E_1|}:X_2\ra\mathbb{P}^1$
and  $\varphi_{|E_2|}:X_2\ra\mathbb{P}^1$ have no reducible fibers
and  8 disjoint sections which can be chosen to generate the
Mordell--Weil group (which is isomorphic to $\Z^7$). One can choose these sections, for  each fibration, in such a way that 7 sections are in common, the eighth
section of $\varphi_{|E_1|}$ is a 5-section for $\varphi_{|E_2|}$
and vice versa the eighth section of $\varphi_{|E_2|}$ is a
5-section for $\varphi_{|E_1|}$. The eight sections of
$\varphi_{|E_1|}$ (resp. $\varphi_{|E_2|}$) chosen as above form
an even set of eight disjoint rational curves, so $X_2$ is a
Nikulin surface and $\NS(X_2)\simeq M_{2,2}$.
\end{proposition}
\proof Since one has a basis of $\NS(X_2)$, one can compute explicitly the sublattice $E_1^{\perp}:=\{D\in \NS(X_2)\simeq L_{2,2}'| DE_1=0
\}$ and one observes that it is $P(2)$ for a certain degenerate even lattice $P$. In
particular there are no $(-2)$-classes orthogonal to $E_1$ in
$\NS(X_2)$ and thus the fibration $\varphi_{|E_1|}$ does not admit
reducible fibers.
The fibration $\varphi_{|E_2|}:X_2\ra\mathbb{P}^1$ is the image of $\varphi_{|E_1|}$ for the automorphism $\sigma$, so also $\varphi_{|E_2|}$ has no reducible fibers too.

To conclude the proof it suffices to exhibit the classes of the irreducible rational curves with the required properties.
Let us assume that $N_i$ is a class such that $N_i^2=-2$, $N_iL>0$, $N_iE_1=1$. Then $N_i$ is the class of an effective divisor (by $N_iL>0$), supported on a (possibly reducible) curve. If $N_i$ is irreducible, then it is a section of $\varphi_{|E_1|}$. Otherwise it should be the sum of a section and some irreducible components of reducible fibers, but there are no reducible fibers in the genus 1 fibration $\varphi_{|E_1|}$. So $N_i$ is a section of $\varphi_{|E_1|}$.
All the classes listed below satisfy the conditions $N_i^2=-2$, $N_iL>0$, $N_iE_1=1$, so they are supported on irreducible rational curves, all sections of $E_1$:
$$
\begin{array}{l}
N_1=E_1+e_2;\ \ N_2=E_1+e_2+e_3;\\ N_3=E_1+e_2+e_3+e_4;\ \ N_4=E_1+e_2+e_3+e_4+e_5;\\
N_5=E_1+e_2+e_3+e_4+e_5+e_6;\ \ N_6=E_1+e_2+e_3+e_4+e_5+e_6+e_7;\\
N_7=E_1-2e_1-3e_2-5e_3-4e_4-3e_5-2e_6-e_7-3e_8;\ \ N_8=3E_1+e_2-e_8.\end{array}$$
Since $N_iN_j=0$ for every $i,j=1,\ldots, 8$, $i\neq j$, the curves $N_i$ are disjoint. Moreover $(\sum_{i=1}^8N_i)/2\in \NS(X_2)$, so $\{N_1,\ldots, N_8\}$ is an even set of disjoint rational curves and thus there is a $2:1$ cover branched on these rational curves, i.e. $X_2\in\mathcal{M}_2$.

The curves $N_i$ $i=1,\ldots 7$ intersect both $E_1$ and $E_2$ in 1 point.  So the fibrations $\varphi_{|E_1|}$ and $\varphi_{|E_2|}$ share 7 sections.

The divisor $H:=6E_1-e_1+2e_2-2e_8$ is a pseudoample divisor of self intersection 4 which is orthogonal to all the $N_i$'s. So $N(X_2)\simeq M_{2,2}$.

The class $N_8'':=3E_1-2e_1+e_2-e_8$ is a section of
$\varphi_{|E_2|}$ and a 5 section of $\varphi_{|E_1|}$. The class
$\left(\sum_{i=1}^7N_i+N_8''\right)/2\in \NS(X_2)$, so
$\{N_1,N_2,N_3,N_4,N_5,N_6,N_7,N_8''\}$ is an even set of disjoint
rational curves.
\endproof
The explicit knowledge of the change of bases  from $\{E_1,e_1,\ldots e_8\}$ to $\{H, N_1,\ldots, N_7, \sum_{i=1}^8 N_i/2\}$ given in the proof of Proposition \ref{prop: L4 simeq M4'} allows one to obtain some
interesting geometric characterizations of the K3 surfaces with
$\NS(X_2)\simeq L_{2,2}'$. Indeed, let $S$ be a K3 surface of Picard
number 9 and which satisfies one of the following conditions:
\begin{itemize}
\item $S$ admits an elliptic fibration $\mathcal{E}:S\ra\mathbb{P}^1$ without reducible fibers and admitting 8 disjoint sections, $P_1,\ldots,P_8$, such that $(\sum_iP_i)/2\in \NS(S)$.
\item $S$ admits an elliptic fibration $\mathcal{E}:S\ra\mathbb{P}^1$ without reducible fibers with zero section $O$. The Mordell--Weil group of $\mathcal{E}$ is generated by 7 sections, $P_1,\ldots,P_7$, such that $\{O,P_1,\ldots, P_6\}$ are mutually disjoint and $P_7$ intersects the zero section in 12 points and the other sections $P_i$, $i=1,\ldots, 6$ in 6 points.
\item $S$ admits two elliptic fibrations $\mathcal{E}$ and $\mathcal{F}$ with class of the fiber $E$ and $F$ respectively such that $EF=2$. Let us assume that there are 7 orthogonal rational curves such that 6 are sections of both the fibrations, the seventh is section of one fibration and a 5-section for the others.
\end{itemize}
Then $S$ satisfies also the other conditions, it admits a symplectic involution switching $\mathcal{E}$ and $\mathcal{F}$ and $S$ is a Nikulin surface. In particular $\NS(X_2)\simeq L_{2,2}'\simeq M_{2,2}$. Indeed, any of the above set of data of fibrations and sections is enough to exhibit the lattice $L_{2,2}'$ as the N\'eron--Severi group of $X_2$, as it follows by the proof of Proposition \ref{prop: L4 simeq M4'}.\\

The map $\varphi_{|H|}:X_2\ra\mathbb{P}^3$ exhibits $X_2$ as a singular quartic in $\mathbb{P}^3$ and its eight nodes are $\varphi_{|H|}(N_i)$ for $i=1,\ldots, 8$. It is well known that the projection of a nodal quartic from a node gives a model of the same K3 surface as a double cover of $\mathbb{P}^2$ branched on a sextic. In particular, let us consider the projection by the node $\varphi_{|H|}(N_8)$, induced by the linear system $|H-N_8|$. We thus have a $2:1$ map $\varphi_{|H-N_8|}:X_2\ra \mathbb{P}^2$, which contracts the 7 curves $N_i$ to seven nodes of the branch sextic. Hence we obtain the following diagram, where the vertical arrows are contractions of 7 curves and the horizontal arrows are $2:1$ maps:
$$\xymatrix{X_2\ar[drr]^-{\varphi_{|H-N_8|}}\ar[rr]^{2:1}\ar[d]&&\widetilde{\mathbb{P}^2}\ar[d]\\
\varphi_{|H|}(X_2)\ar[rr]^{2:1}&&\mathbb{P}^2}$$
In particular $\widetilde{\mathbb{P}^2}$ is the blow up of $\mathbb{P}^2$ in the seven points $\varphi_{|H-N_8|}(N_i)$ (which are the singular points of the branch locus of the map $X_2\ra\mathbb{P}^2$) and thus $\widetilde{\mathbb{P}^2}$ is a del Pezzo surface of degree 2. The cover involution of the double cover $X_2\ra\widetilde{\mathbb{P}}^2$ is an involution $i$, such that $i^*(H-N_8)=H-N_8$, $i^*(N_i)=N_i$, $i=1,\ldots 7$ and $\iota^*(N_8)=2H-3N_8$. One is now able to rewrite the action of $i^*$ on the basis $\{L,e_1\ldots, e_8\}$ and one finds that $i^*=\iota_{dP}^*$ (where $\iota_{dP}^*$ is as in Proposition \ref{prop: three involutions on NS(X2)}). Thus, we the notation of \eqref{eq: bidouble  P2}, one obtains $\widetilde{\mathbb{P}^2}=dP$, $\iota_{dP}=i$ and the map $\pi_{dP}$ is induced by the projection of $\varphi_{|H|}(X_2)$ from the node $\varphi_{|H|}(N_8)$.
The even set $\{N_1,\ldots, N_7, N_8''\}$ is nothing but the image of the even set $\{N_1,\ldots, N_8\}$ for the action of $\iota_{dP}$.\\

In Section \ref{subsec: Galois closure 2^2:1 cover} we proved that
the construction of the $\mathcal{D}_4$ Galois cover of $X_d$ is
totally determined by two sets of rational curves in $X_d$, i.e.
the sets $\{R_1,\ldots, R_8\}$ and $\{\overline{N}_1,\ldots,
\overline{N}_8\}$. In particular in the case we are
now considering, i.e. if $d=1$, the curves $R_i$, $i=1,\ldots, 8$ were already considered in the diagram \eqref{eq: bidouble  P2} and are mapped by $q_3:X_1\ra\mathbb{P}^2$ to the eight singular points of the branch sextic $B\cap C_0$. We now describe the curves $\overline{N}_i$, by giving their image as plane curves $q_3(\overline{N_i})\subset\mathbb{P}^2$. By construction, $q_3(\overline{N_i})=\pi(N_i)$, where $\pi:X_2\dashrightarrow\mathbb{P}^2$ is the rational
$2^2:1$ map given in \eqref{eq: bidouble  P2}.

\begin{proposition}\label{prop: palne curves image of Ni}
For each $i=1,\ldots, 7$, the curve $\pi(N_i)\subset\mathbb{P}^2$ is a bitangent line to
the quartic $B$. The curve $\pi(N_8)$ is a
rational irreducible sextic $D\subset\mathbb{P}^2$ which is
tangent to $B\cup C_0$ in all their intersection points.
The curves $\pi(N_i)$, $i=1,\ldots, 8$ split in $X_1$, the orbit of $N_i$ with respect to $\langle\iota_Q, \iota_{dP} \rangle$ consists of two rational curves if $i=1,\ldots, 7$ and of the four curves if $i=8$.
\end{proposition}
\proof
The surface $dP$ is a degree 2 del Pezzo surface, and then it is naturally endowed with an involution $i$, which is the cover involution of the $2:1$ map $dP\ra\mathbb{P}^2$ (see \cite[Chaptes VII, Section 4]{DO} for details on del Pezzo surfaces of degree 2 and its involution). The double cover $q_2:dP\ra\mathbb{P}^2$ is branched on $B\subset\mathbb{P}^2$. Since $dP$ is a del Pezzo
surface of degree 2, there is a set
of 7 disjoint $(-1)$-curves on $dP$, denoted by $p_i$,
$i=1,\ldots, 7$ and such that $\beta_{dP}:dP\ra\mathbb{P}^2$ is a contraction of these $(-1)$-curves. The plane curves $q_2(p_i)\subset\mathbb{P}^2$ are 7 lines which are bitangent to $B$, and each of them splits in the double cover into two rational curves $p_i$, $i(p_i)$, $i=1,\ldots,7$. So we have the following commutative diagram:
\begin{align}\label{eq: diagram dP P2 X2}\xymatrix{X_2\ar[drr]^-{\varphi_{|H-N_8|}}\ar[d]_{\beta_{X_2}}\ar[rr]^{2:1}&&dP\ar[d]^{\beta_{dP}}\ar[rr]_-{q_2}^-{2:1}&&\mathbb{P}^2\supset B\cup C_0\\
\varphi_{|H|}(X_2)\ar[rr]_{2:1}&&\mathbb{P}^2}\end{align}
where $\beta_{X_2}$ contracts the curves $N_i$, $i=1,\ldots, 8$.
For $i=1,\ldots, 7$, one has $\varphi_{|H-N_8|}\left(N_i\right)=\beta_{dP}(p_i)$. Each of the 7 lines $q_2(p_i)\subset{\mathbb{P}}^2$ is bitangent to $B$, and intersects $C_0$ transversally. So $q_2(p_i)$ does not split in the double cover $q_1:Q\ra\mathbb{P}^2$, which is branched on $C_0$. In particular, $q_1^{-1}(q_2(p_i))=q_1^{-1}(\pi(N_i))$ is an irreducible smooth rational curve for $i=1,\ldots , 7$. So, for each
$i=1,\ldots, 7$, $\pi^{-1}(\pi_{dP}(p_i))\subset X_2$ consists of
a pair of rational curves, switched by
$\iota_{Q}$, preserved by
$\iota_{dP}$ and then switched by
$\sigma=\iota_{Q}\circ\iota_{dP}$. This can also be checked directly on the classes of the curves $N_i$ by using Propositions \ref{prop: three involutions on NS(X2)} and \ref{prop: L4 simeq M4'}, indeed $\iota_Q^*(N_i)=N_i$ and $\iota_{dP}(N_i)=\sigma(N_i)\neq N_i$ for $i=1,\ldots, 7$.
Since for each $i\in\{1,\ldots, 7\}$ the curve $N_i$ is a
section of both the elliptic fibrations $|E_1|$ and $|E_2|$ on $X_2$, the curve $q_1(N_i)\subset Q\simeq \mathbb{P}^1\times\mathbb{P}^1$ is a curve of bidegree $(1,1)$ if $i=1,\ldots 7$.

It remains to describe the curve $N_8\subset X_2$. The orbit of
$N_8$ is given by the four classes $N_8=3E_1+e_2-e_8$,
$N_8':=\sigma(N_8)=3E_1-3e_1-e_2+e_8$;
$N_8''=\iota_{dP}(N_8)=3E_1-2e_1+e_2-e_8$,
$N_8''':=\sigma(\iota_{dP}(N_8))=3E_1-e_1-e_2+e_8$. Since
$N_8+N_8'+N_8''+N_8'''\simeq 12E_1-6e_1\simeq 6L$ the curve
$\pi(N_8)$ is a sextic $C_8$ in $\mathbb{P}^2$, which splits in
all the double covers $Q$, $dP$ and $X_2$ of $\mathbb{P}^2$. The
sextic $C_8$ is a rational curve (since it is the image of
rational curves) and thus has 10 nodes. Moreover, since $C_8$
splits in the double covers, $C_8\cap C_0$ consists of 6 points
with multiplicity two and $C_8\cap B$  consists of 12 points with
multiplicity two. The inverse image of $C_8$ in $Q$ consists of
two rational curves, one of bidegree $(1,5)$ which is the common
image of $N_8$ and $N_8'''$ and one of bidegree $(5,1)$ which is
the common image of $N_8'$ and $N_8''$. The bidegrees of these
curves are obtained by the fact that $N_8$ is a section of $E_1$
and a 5-section of $E_2$, and $N_8'$ is a section of $E_1$ and a
5-section of $E_2$. The inverse image $q_3^{-1}(C_8)$ in $X_1$
consists of two rational curves, one is $\pi_2(N_8)=\pi_2(N_8')$
and it is the curve denoted by $\overline{N_8}$ in Section
\ref{subsub:  bidouble cover Xd}, the other is
$\pi_2(N_8'')=\pi_2(N_8''')$.\endproof

\begin{rem}{\rm
As in the proof of Proposition \ref{prop: palne curves image of Ni} one is able to determine the image of the curves in the linear systems $|E_i|$ under the map $\pi$. The orbit of $E_1$ for $\langle\iota_Q^*, \iota_{dP}^*\rangle$ is $\{E_1, E_2\}$, thus we have two elliptic fibrations which are switched by $\sigma$ and by $\iota_{dP}$ but each of them is preserved by $\iota_Q$. Since $E_1+E_2=L$, a curve $F_1\in |E_1|$ is mapped to a line $f_1$ in $\mathbb{P}^2$. Moreover, for a general $F_1$, $q_1^{-1}(f_1)$ is the union of the two curves $\pi_Q(F_1)$ and $\pi_Q(\sigma(F_1))$. Hence the line $f_1$ is tangent to the conic $C_0$ (which is the branch locus of $q_1:Q\ra\mathbb{P}^2$). The line $f_1$ does not splits for the covers $q_2$ and $q_3$ and in particular the class $(q_3)_*(E_1)$ induces an elliptic fibration on $X_1$. So $q_3^{-1}(f_1)$ is a genus 1 curve. This implies that $q\cap B$ consists of 4 disjoint points (which are the branch points of the $2:1$ cover $q_3^{-1}(f_1)\ra f_1$). Hence the 1-dimensional linear system of genus 1 curve in $|E_1|$ is mapped by $\pi$ to the 1-dimensional linear system of lines tangent to the conic $C_0$. The same holds true for the 1-dimensional linear system $|E_2|$, since $\sigma^*(E_1)=E_2$.

By definition the $2:1$ map $q_2:dP\ra\mathbb{P}^2$ is the anticanonical map, and then, denoted by $h$ the class of a line in $\mathbb{P}^2$, $q_2^*(h)=-K_{dP}$. So $q_2^{-1}(f_1)$ is a genus 1 curve in the anticanonical system, with the special property that it intersects the curve $q_2^{-1}(C_0)$ with even multiplicity in each of their intersection points. Hence the curve $q_2^{-1}(f_1)\subset dP$ splits in the double cover $\pi_{dP}:X_2\ra dP$. With the notation of \eqref{eq: diagram dP P2  X2}, the curve $\beta_{dP}(q_2^{-1}(f_1))$ is a plane cubic tangent to $\beta_{dP}(q_2^{-1}(C_0))$.}\end{rem}

Until now our point of view was to consider $X_2$ as a surface with a symplectic automorphism $\sigma$ and to determine its structure as Nikulin surface, but one can consider the reverse problem: given a Nikulin surface with N\'eron--Severi group $M_{2,2}$, it has a very natural model as quartic in $\mathbb{P}^3$ with 8 nodes. To reconstruct the structure of this surface as double cover of $\mathbb{P}^1\times\mathbb{P}^1$ admitting a symplectic involution one has to identify the two elliptic fibrations $\varphi_{|E_1|}:X_2\ra\mathbb{P}^1$ and $\varphi_{|E_2|}:X_2\ra\mathbb{P}^1$.
We gave a change of basis from $\{E_1,e_1,\ldots, e_8\}$ to $\{H, N_1,\ldots, N_7, \sum_{i=1}^8N_1/2\}$ in proof of Proposition \ref{prop: L4 simeq M4'}. Its inverse allows us to find the class of $E_1$ in terms of the classes $H$ and $N_i$, $i=1,\ldots, 8$, in particular $E_1=H-(\sum_{i=1}^8N_i)/2$. The curves in this linear system are mapped to cubics by the linear system $|H-N_8|$. So, given a curve $F_1\in |E_1|$, $c:=\varphi_{|H-N_8|}(F_1)$ is a cubic curve in $\mathbb{P}^2$ and $\varphi_{|H-N_8|}^{-1}(c)$ consists of two curves, whose linear systems are $E_1=H-(\sum_{i=1}^8N_i)/2$ and $E_2=2H-(\sum_{i=1}^8N_i)/2-2N_8$ respectively. Their sum exhibits $S$ as double cover of $\mathbb{P}^1\times \mathbb{P}^1$ admitting the required symplectic involution.

In \cite[Section 3.7]{vGS} an equation of $X_2$ is given, starting from a description of a K3 surface $X_4$ such that $\NS(X_4)=L_{4,2}'$. The surface $X_4$ is given as complete intersections of three quadrics in $\mathbb{P}^5_{(y_0:y_1:x_0:\ldots: x_3)}$ of the form \begin{eqnarray}\label{eq: X_4}\left\{\begin{array}{l}y_0^2=Q_1(x_0:x_1:x_2:x_3)\\
y_1^2=Q_2(x_0:x_1:x_2:x_3)\\y_0y_1=Q_3(x_0:x_1:x_2:x_3)\end{array}\right.\end{eqnarray}
Each complete intersection with equation \eqref{eq: X_4} admits a
symplectic involution induced by the projective transformation
$$(y_0:y_1:x_0:x_1:x_2:x_3)\mapsto (-y_0:-y_1:x_0:x_1:x_2:x_3).$$
As shown in \cite[Section 3.7]{vGS}, a singular model of the
quotient surface is given by \begin{equation}\label{eq: Y2 as
quartic}Q_1(x_0:x_1:x_2:x_3)Q_2(x_0:x_1:x_2:x_3)=Q_3^2(x_0:x_1:x_2:x_3)\subset\mathbb{P}^3_{(x_0:x_1:x_2:x_3)}.\end{equation}
By Proposition \ref{prop: NS(Y) determines NS(X)}, the smooth
model of the quotient surface \eqref{eq: Y2 as  quartic} has
N\'eron--Severi group isometric to $M_{2,2}$, i.e. it is the
surface $X_2\simeq S$ and the map to
$\mathbb{P}^3_{(x_0:x_1:x_2:x_3)}$ is given by the linear system
of the pseudo ample polarization $H$ (with the notation of
Proposition \ref{prop: L4 simeq M4'}).

Let us consider the pencil of quadrics
$\mathcal{P}_t:=\{Q_1=tQ_3\}\subset\mathbb{P}^3$. It cuts on $X_2$
a pencil of curves, whose class is $2H-\sum_{i=1}^8N_i$, since all
the quadrics in $\mathcal{P}_t$ pass through the 8 points in $Q_1\cap
Q_2\cap Q_3$, which are the singular points of the surfaces in
\eqref{eq: Y2 as  quartic}. For almost every $t$, $\mathcal{P}_t$
cuts two genus
1-curves on the surfaces in \eqref{eq: Y2 as  quartic}: one is the complete intersection $Q_1\cap Q_2$ (and does
not depend on $t$) the other is $(Q_1-tQ_3)\cap (Q_3-tQ_2)$. So,
the first curve is a fixed component of the linear system
$2H-\sum_{i=1}^8N_i$, the latter is a movable curve. The curves
$Q_1\cap Q_2$ and $(Q_1-tQ_3)\cap (Q_3-tQ_2)$ intersect
transversally in the singular points of the quartic \eqref{eq: Y2
as  quartic}. So, they have no intersection points in the blow up
$X_2$ of the quartic \eqref{eq: Y2 as  quartic} in its singular
points. Hence the curves $Q_1\cap Q_2$ and $(Q_1-tQ_3)\cap
(Q_3-tQ_2)$ are two fibers of the same fibration $X_2\rightarrow
\mathbb{P}^1_t$ and are represented by the same divisor in
$\NS(X_2)$. It is necessarily
$\left(2H-\sum_{i=1}^8N_i\right)/2=H-\left(\sum_{i=1}^8N_i\right)/2$.
This is the divisor $E_1$ considered above so we conclude
that if the surface $S\simeq X_2$ is embedded in $\mathbb{P}^3$ as
a quartic of the form $Q_1Q_2=Q_3^2$, then the elliptic fibration
$E_1$ is cut out by
$\mathcal{P}_t:=\{Q_1=tQ_3\}$. The elliptic fibration $E_2$ is the image of $E_1$ under $\iota_{dP}$, which is the involution induced by the projection from the node $\varphi_{|H|}(N_8)$ of $\varphi_{|H|}(X_2)$.

\subsection{Special 10-dimensional subfamilies of $\mathcal{L}_2$ and $\mathcal{M}_2$}\label{subsec: U+N e U+E8}

In Proposition \ref{prop: properties of the U+Mn polarized K3} we
discussed the family $\mathcal{U}_n$ of $(U\oplus
\mathbb{M}_n)$-polarized K3 surfaces, proving that it is contained
in $\mathcal{L}_n\cap \mathcal{M}_n$ and it has codimension 1 in
this space. This holds for every admissible $n$, so in particular
for $n=2$. Here we reconsider this family, since it also has
interesting properties with respect to the components of
$\mathcal{M}_2$: it is contained in the common intersection of all
the irreducible components of $\mathcal{M}_2$. We discuss the
analogous property for the components of $\mathcal{L}_2$,
identifying another interesting family of K3 surfaces, which has
codimension 1 in each component of $\mathcal{L}_2$.
More precisely the aim of this section is to prove the following:
\begin{itemize}
\item There exists an irreducible connected 10-dimensional subvariety of the moduli space of  K3 surfaces (it is $\mathcal{U}_2$) which is properly contained in all the families of Nikulin surfaces.
Moreover all the K3 surfaces in this subvariety also admit a
symplectic involution.
\item
There exists an irreducible connected 10-dimensional subvariety of the moduli space of  K3 surfaces which is properly contained in all the families of K3 surfaces admitting a symplectic involution. All the K3 surfaces in this subvariety are also Nikulin surfaces.
\end{itemize}

\begin{proposition}\label{prop: U+N polarized K3}
There exists an overlattice of index 2 of $U(2)\oplus N$, denoted
by $(U(2)\oplus N)'$, which is isometric to $U\oplus N$ and such
that for any $d\in\mathbb{N}_{\geq 1}$, both the lattice $M_{d,2}$
and the lattice $M_{2d,2}'$ are primitively embedded in
$(U(2)\oplus N)'$. Hence all the irreducible components of the
11-dimensional families of Nikulin surfaces properly contain the
10-dimensional family $\mathcal{U}_2=\mathcal{P}(U\oplus N)$.
\end{proposition}
\proof Let us consider the lattice $U(2)\oplus N$. Let $u_1$ and
$u_2$ be the basis of $U(2)$ such that $u_j^2=0$, $j=1,2$ and
$u_1u_2=2$, and let $w_{i,j}$, $i=1,2$ , $j=1,2,3$, be a set of vectors in
$N$ such that $w_{i,h}/2$ are contained in the discriminant group
of $N\subset U(2)\oplus N$. Moreover we assume that the
discriminant form on $w_{i,j}/2$, $i=1,2$, $j=1,2,3$ is $u(2)^3$.
The vector $v:=\left(u_{1}+u_{2}+w_{1,1}+w_{2,1}\right)/2$ is
isotropic in $A_{U(2)\oplus N}$, and the lattice obtained by adding the vector $v$ to $U(2)\oplus N$ is an even overlattice of index 2
of $U(2)\oplus N$. Let us call it $\left(U(2)\oplus N\right)'$.
The discriminant group of $\left(U(2)\oplus N\right)'$ is
generated by $\left(u_{1}+w_{1,1}\right)/2$,
$\left(u_{1}+w_{1,2}\right)/2$, $w_{i,j}$, $i=1,2$, $j=2,3$ and
its discriminant form is $u(2)^3$. There is a unique, up to
isometry, even hyperbolic lattice with rank 10, length 6 and
prescribed  discriminant form $u(2)^3$. Hence $\left(U(2)\oplus
N\right)'\simeq U\oplus N$.

To give a primitive embedding of $M_{d,2}\simeq \langle
2d\rangle\oplus N$ in $U\oplus N$ it suffices to give a primitive
embedding of $\langle 2d\rangle$ in $U$, for example the embedding
$\left(\begin{array}{c}1\\d\end{array}\right)\hookrightarrow U$ is
primitive.

To give a primitive embedding of $M_{2d,2}'$ in $\left(U(2)\oplus
N\right)'$ we consider a primitive embedding of $M_{2d,2}$ in
$U(2)\oplus N$, which extends primitively to their overlattices.
As above, a primitive embedding of $M_{2d,2}\simeq\langle
4d\rangle\oplus N$ in $U(2)\oplus N$ is induced by a primitive
embedding of $\langle 4d\rangle$ in $U(2)$. We fix this embedding
to be $\langle u_1+du_2\rangle\hookrightarrow U(2)$. This induces
a primitive embedding of $M_{2d,2}'$ in $\left(U(2)\oplus
N\right)'\simeq U\oplus N$.\endproof

Let $S$ be a K3 surface with $\NS(S)\simeq U\oplus N$.
Then $S$ admits an elliptic fibration with 8 reducible fibers of
type $I_2$ and a 2-torsion section. We denote by $F$ the class of
the fiber of this fibration, by $O$ the class of the zero section,
by $t$ the 2-torsion section and by $C_i^j$, $i=0,1$, $j=1,\ldots,
8$ the $i$-th component of the $j$-th fiber (with the usual
assumption that the 0-component meets the zero section). A basis
of $U\oplus N$ is then given by $F$, $F+O$, $C_1^j$, $j=1,\ldots
7$, $\left(\sum_{j=1}^8C_1^j\right)/2=2F+O-t$. The translation by
a 2-torsion section is a symplectic involution, denoted by
$\sigma$ and classically called van Geemen--Sarti involution. Its
action is $F\leftrightarrow F$, $O\leftrightarrow t$,
$C_1^j\leftrightarrow C_0^j$. The sublattice of $\NS(S)$ invariant
for $\sigma$ is $\NS(S)^{\sigma}\simeq \langle F,s+t\rangle\simeq
U(2)$. This exhibits the N\'eron--Severi group $\NS(S)\simeq
U\oplus N$ as an overlattice (necessarily of index $2^2$) of
$U(2)\oplus E_8(-2)$, since
$\left(\NS(S)^{\sigma}\right)^{\perp}\simeq E_8(-2)$. Chosen a
positive integer $e$, the divisor $v:=F-e(s+t)$ has the following
properties: $v^2=-4e$, $v$ is invariant and
$v^{\perp_{\NS(S)}}\simeq L_{2e,2}'$. In particular the van
Geemen--Sarti involution on $S$ induces the symplectic involution
whose action on $\NS(S)$ is $-1$ on
$E_8(-2)\hookrightarrow L_{2e,2}'\simeq v^{\perp}$ and $+1$ on
$v$.  Hence the isometry $\sigma^*$ on $\NS(S)$ extends the isometry associated to the symplectic involution on $L_{2e,2}'$-polarized K3 surfaces, once an embedding $L_{2e,2}'\hookrightarrow (U\oplus N)$ as in the proof of Proposition \ref{prop: U+N polarized K3} is fixed.

Now we consider the analogous problem on the irreducible
components of $\mathcal{L}_n$.
\begin{proposition}\label{prop: U+E8(2)-polarized K3}
The 10-dimensional family $\cP(U\oplus E_8(-2))$ is properly contained in all the families
$\cP(L_{e,2})$ and $\cP(L_{2e,2}')$.

The lattice $U\oplus E_8(-2)$ is
isometric to the lattice $U(2)\oplus N$, hence all the members of
the family $\cP(U\oplus E_8(-2))$ are Nikulin surfaces.
\end{proposition}
\proof The primitive embedding of $L_{e,2}\simeq \langle
2e\rangle\oplus E_8(-2)$ in $U\oplus E_8(-2)$ is induced by the
primitive embedding of $\langle 2e\rangle\simeq \left\langle
\left(\begin{array}{c}1\\e\end{array}\right)\right\rangle$ in $U$,
as in the proof of Proposition \ref{prop: U+N polarized K3}. We
observe that $U\oplus E_8(-2)$ is an overlattice of index 2 of
$U(2)\oplus E_8(-2)$. Indeed, similarly to what we did in proof of
Proposition \ref{prop: U+N polarized K3}, we consider the basis
$u_1$ and $u_2$ of $U(2)\subset U(2)\oplus E_8(-2)$ and the
vectors $w_{i,j}/2$ $i=1,2$, $j=1,2,3,4$ in $A_{U(2)\oplus
E_8(-2)}$ such that the discriminant form on $u_1/2$, $u_2/2$ and
$w_{i,j}/2$, $i=1,2$, $j=1,\ldots, 4$ is $u(2)^5$. By adding
$v=(u_{1}+u_{2}+w_{1,1}+w_{2,1})/2$ to $U(2)\oplus E_8(-2)$ one
obtains an even overlattice $\left(U(2)\oplus E_8(-2)\right)'$ of
index 2 of $U(2)\oplus E_8(-2)$, which is isometric to $U\oplus
E_8(-2)$. Hence the primitive embedding of $L_{2e,2}'$ in
$\left(U(2)\oplus E_8(-2)\right)'\simeq U\oplus E_8(-2)$ is
induced by a primitive embedding of $\langle 4e\rangle$ in $U(2)$,
which is given by   $\langle 4e\rangle\simeq \left\langle
\left(\begin{array}{c}1\\e\end{array}\right)\right\rangle$ in
$U(2)$.

Since $L_{e,2}$ and $L_{2e,2}'$ are primitively embedded in
$U\oplus E_8(-2)$ and they determine uniquely their orthogonal complement in
$\Lambda_{K3}$, the families  $\mathcal{P}(L_{e,2})$ and $\mathcal{P}(L_{2e,2}')$ properly
contain the family $\mathcal{P}(U\oplus E_8(-2))$.

The isometry between the lattices  $U\oplus E_8(-2)$ and
$U(2)\oplus N$ follows by observing that they are lattices with
rank 10, length 8 and the same discriminant form.
\endproof

Let $\mathcal{E}_R:R\ra\mathbb{P}^1$ be a rational elliptic
surface (i.e. $R$ is a rational surface endowed with an elliptic
fibration $\mathcal{E}_R$). It is known that a base change of
order 2 on this elliptic fibration branched on two smooth fibers
induces an elliptic fibration $\mathcal{E}_S:S\ra\mathbb{P}^1$ on
a K3 surface $S$. If the fibration $\mathcal{E}_R$ has no
reducible fibers, then $\NS(S)\simeq U\oplus E_8(-2)$, see e.g.
\cite[Proposition 4.6]{GSal}. More in general the family of the K3
surfaces obtained by a base change of order 2 on a rational
elliptic surface, is the family $\mathcal{P}\left(U\oplus
E_8(-2)\right)$, see e.g. \cite{GSal}.

\begin{proposition}\label{prop: sympl inv on U+E8 polarized}
The 10-dimensional family $\cP(U\oplus E_8(-2))$ is the family
$\mathcal{R}$ of the K3 surfaces obtained by a base change of
order two on a rational elliptic fibration
$\mathcal{E}_R:R\ra\mathbb{P}^1$. Let $S$ be a general member of
$\mathcal{R}$ and let $\mathcal{E}_S$ be the elliptic fibration
induced by $\mathcal{E}_R$: $\mathcal{E}_S$ has no reducible
fibers and its Mordell--Weil rank is equal to 8. The symplectic
involution $\sigma$ on $S$ preserves $\mathcal{E}_S$. Denoted by
$\widetilde{S/\sigma}$ the desingularization of $S/\sigma$,
$\NS(\widetilde{S/\sigma})\simeq U\oplus D_4\oplus D_4$.
\end{proposition}
\proof We already observed that $S$ is obtained by a base change
of order 2 by $\mathcal{E}_R:R\ra\mathbb{P}^1$. Then
$\mathcal{E}_S$ admits an involution $\iota$ which acts only on
the basis of the fibration, and which is the deck involution of
the generically 2:1 cover $R\ra S$. The involution $\iota$ maps
fibers of $\mathcal{E}_S$ to other fibers and in particular
preserves the class of the fiber and of the sections, i.e. it acts
trivially on the N\'eron--Severi group. Thus it preserves the
elliptic fibration (cf. \cite[Proposition 4.6]{GSal}). Moreover,
$\iota$ preserves exactly two fibers of $\mathcal{E}_S$ (the
ramification fibers of the cover $R\ra S$). The elliptic fibration
$\mathcal{E}_S$ is preserved also by the elliptic involution
$\epsilon$, which preserves the classes of the fiber and of the
zero section (i.e. a set of generators of $U$ in $\NS(S)\simeq
U\oplus E_{8}(-2)$). The composition $\sigma:=\iota\circ\epsilon$
is a symplectic involution which acts trivially on
$U\hookrightarrow U\oplus E_8(-2)\simeq \NS(S)$ and as $-\mathrm{id}$ on $E_8(-2)\hookrightarrow U\oplus E_8(-2)\simeq \NS(S)$.
Thus $\sigma$ is a symplectic involution whose fixed locus
consists of 4 points on each of the two fibers preserved by
$\iota$. Hence the elliptic fibration
$\mathcal{E}_S:S\ra\mathbb{P}^1$ induces an elliptic fibration on
$\widetilde{S/\sigma}$ whose generic fiber is a copy of the two fibers
of $\mathcal{E}_S$ switched by $\sigma$. The images of the two
fibers preserved by $\iota$ are two fibers of type $I_0^*$. The
Picard number of $\widetilde{S/\sigma}$ is 10, which is also the rank
of the trivial lattice of an elliptic fibration with two fibers of
type $I_0^*$. We conclude that there are no sections of infinite
order for the elliptic fibration induced by $\mathcal{E}_S$ on
$\widetilde{S/\sigma}$ and that $\NS(\widetilde{S/\sigma})\simeq
U\oplus D_4\oplus D_4$.
\endproof

We observe that $U\oplus D_4\oplus D_4\not\simeq U\oplus E_8(-2)$
since their discriminant groups are different, so $\NS(S)\not\simeq
\NS(\widetilde{S/\sigma})$.

By Proposition \ref{prop: U+E8(2)-polarized K3}, if $S$ is a K3
surface such that $\NS(S)\simeq U\oplus E_8(-2)$, then it admits a
symplectic involution (described in the proof of Proposition
\ref{prop: sympl inv on U+E8 polarized}) and it is also $2:1$
cyclically covered by a K3 surface. So it admits a 2-divisible set
of rational curves, which we describe here. As observed $S$ is
obtained by a base change of order 2 on $R$. Since $R$ is a
rational elliptic surface, it is the blow up of $\mathbb{P}^2$ in
nine points which are the base points of a pencil of generically
smooth cubics. So $S$, which is a 2:1 double cover of $R$ branched
on two smooth fibers, is a generically $2:1$ cover of
$\mathbb{P}^2$ branched in the union of two smooth cubics $C_1$
and $C_2$ (see e.g. \cite[Section 2.2]{GSal}). The branch locus is
singular in the nine points $C_1\cap C_2$. We denote by $H$ the
genus $2$ divisor on $S$ such that
$\varphi_{|H|}:S\ra\mathbb{P}^2$ is this $2:1$ cover of
$\mathbb{P}^2$ and by $D_i$, $i=0,\ldots 8,$ the classes of the
rational curves contracted by $\varphi_{|H|}$ to the nine singular
points of the branch locus. By construction the fiber of the
fibration $\mathcal{E}_S$ is the class of $C_1$ (and of $C_2$),
i.e. $\left(3H-\sum_{i=0}^8D_i\right)/2$. The curves in the linear
system $|H-D_0|$ (and in $|H-D_1|$ respectively) on $S$ are mapped
to lines of a pencil in $\mathbb{P}^2$, with base point
$\varphi_{|H|}(D_0)$ (with base point $\varphi_{|H|}(D_1)$
respectively). Each line in this pencil meets the branch in 4
points (with the exception of $\varphi_{|H|}(D_0)$), and so its
inverse image in $S$ is a $2:1$ cover of a rational curve branched
in 4 points. So the curves in $|H-D_0|$ (resp. $|H-D_1|$) are
genus 1 curves and $|H-D_0|$ and $|H-D_1|$ induce two genus 1
fibrations on $S$ (see \cite[Proposition 3.8]{GSal}). Since
$(H-D_0)(H-D_1)=2$, the map $\varphi_{|2H-D_0-D_1|}$ is a
generically $2:1$ map to
$\mathbb{P}^1\times\mathbb{P}^1\subset\mathbb{P}^3$ (see
\cite{SD}). It contracts the 8 mutually disjoint rational curves
$D_i$, $i=2,\ldots ,8$, and $H-D_0-D_1$. The last contracted curve
is the pullback of the line through the two points
$\varphi_{|H|}(D_0)$ and $\varphi_{|H|}(D_1)$. The $2:1$ map
$\varphi_{|H-D_0|}\times\varphi_{|H-D_1|}:S\ra\mathbb{P}^1\times\mathbb{P}^1$
is induced by the $2:1$ map $\varphi_{|H|}:S\ra\mathbb{P}^2$ via
the birational transformation
$\beta:\mathbb{P}^2\ra\mathbb{P}^1\times \mathbb{P}^1$, which is
the blow up of $\mathbb{P}^2$ in the two points
$\varphi_{|H|}(D_0)$ and $\varphi_{|H|}(D_1)$ followed by the
contraction of the line through these points. So the branch locus
of the $2:1$ cover $S\ra\mathbb{P}^1\times\mathbb{P}^1$ splits in
the union of two genus 1 curves of bidegree $(2,2)$, which are the
images of the two cubics $C_1$ and $C_2$ under the birational
transformation $\beta$. The curves $\beta(C_1)$ and $\beta(C_2)$
intersect in 8 points in $\mathbb{P}^1\times\mathbb{P}^1$, which
are the images of the curves $D_i$, $i=2,\ldots, 8$, and
$H-D_0-D_1$. The classes of the pullback of $\beta(C_1)$ and
$\beta(C_2)$ on $S$ coincide and each of them is represented by
the class
$\left(2(H-D_0)+2(H-D_1)-\sum_{i=2}^8D_i-(H-D_0-D_1)\right)/2$. So
the set of curves $\{D_2,\ldots D_8, H-D_0-D_1\}$ is an even set.


\begin{thebibliography}{PSSW}
\bibitem[BSV]{BSV}  S.\ Boissi\`ere, A.\ Sarti, D.C.\ Veniani, {\it On prime degree isogenies between K3 surfaces}, Rend. Circ. Mat. Palermo (2) {\bf 66} (2017), 3--18.
\bibitem[Bu]{Buskin} N.\ Buskin, {\it Every rational Hodge isometry between two K3 surfaces is algebraic}, Journal f\"ur die reine und angewandte Mathematik (2017), doi:10.1515/crelle-2017-0027.
\bibitem[C]{Catanese} F.\ Catanese, {\it
On the moduli spaces of surfaces of general type},
J. Differential Geom. {\bf 19} (1984), 483--515.
\bibitem[DO]{DO} I.\ Dolgachev, D, Ortland {\it Point sets in projective spaces and theta functions},
Ast\'erisque {\bf 165} (1988).
\bibitem[G1]{G2} A.\ Garbagnati, {\it K3 surfaces with abelian and dihedral groups of symplectic automorphisms}, Comm. Algebra {\bf 41} (2013), 583--616.
\bibitem[G2]{G} A.\ Garbagnati,  {\it On K3 surface quotients of K3 or Abelian surfaces}, Canad. J. Math. {\bf 69} (2017), 338--372.
\bibitem[GSal]{GSal} A.\ Garbagnati, C.\ Salgado, {\it Linear systems on rational elliptic surfaces and elliptic fibrations on K3 surfaces}, J. Pure Appl. Algebra {\bf 223} (2019), 277--300.
\bibitem[GSar1]{GS} A.\ Garbagnati, A.\ Sarti, {\it Projective models of K3 surfaces with an even set}, Adv. Geometry {\bf 8} (2008), 413--440.
\bibitem[GSar2]{GSprime} A.\ Garbagnati,A.\ Sarti {\it Symplectic automorphisms of prime order on K3 surfaces}, J. Algebra {\bf 318} (2007), 323--350.
\bibitem[GSar3]{GSnonprime} A.\ Garbagnati, A.\ Sarti,  {\it Elliptic fibrations and symplectic automorphisms on K3 surfaces}, Comm. Algebra {\bf 37} (2009), 3601--3631.
\bibitem[vGS]{vGS} B.\ van Geemen, A.\ Sarti, {\it Nikulin involutions on K3 surfaces}, Math. Z. {\bf 255} (2007), 731--753.
\bibitem[H]{Hashimoto} K.\ Hashimoto, {\it Finite symplectic actions on the K3 lattice}, Nagoya Math. J., Volume {\bf 206} (2012), 99--153.
\bibitem[Huy]{Huybrechts} D.\ Huybrechts, {\it Motives of isogenous K3 surfaces},  to appear in Commentarii Mathematici Helvetici (2017).
\bibitem[I]{I} H. Inose, {\it Defining equations of singular K3 surfaces and a notion of isogeny}. Proc. int. Symp. on algebraic geometry, Kyoto, 495--502, 1977.
\bibitem[L]{Laterveer} R.\ Laterveer, {\it A family of K3 surfaces having finite-dimensional motive}, Arch. Math. (Basel) {\bf 106} (2016), no. 6, 515--524.
\bibitem[MM]{MirMorr} R.\ Miranda, D.R.\ Morrison, {\it Embeddings of Integral Quadratic Forms}, preprint available at https://www.math.colostate.edu//~miranda/preprints/eiqf.pdf
\bibitem[Mo]{Morrison} D.R. Morrison, {\it On K3 surfaces with large Picard number}, Invent. Math. {\bf 75} (1984) 105--121.
\bibitem[M]{Mukai} S. Mukai, {\it On the moduli space of bundles on K3 surfaces}, I. In: Vector Bundles on Algebraic Varieties, Bombay (1984), 341--413.
\bibitem[N1]{NikulinIntQuadForms} V.V. Nikulin, {\it Integer symmetric bilinear forms and some of their geometric applications}, Izv. Akad. Nauk SSSR Ser. Mat. {\bf 43} (1979), no. 1, 111--177.
\bibitem[N2]{NikSympl} V.V. Nikulin, {\it Finite automorphism groups of K\"ahler K3 surfaces}. Trans. Moscow Math. Soc. {\bf 38} (1980), 71--135.
\bibitem[N3]{NikulinCorresp} V. Nikulin, {\it On correspondences between surfaces of K3 type}, Math. USSR-Izv. 30 (1988), 375--383.
\bibitem[P]{P}  C. Pedrini, {\it Bloch's conjecture and valences of correspondences for K3
surfaces}, arXiv:1510.05832v1.
\bibitem[SD]{SD} B. Saint-Donat, {\it Projective Models of K3 surfaces}, Amer. J. of Math. {\bf 96} (1974) 602--639.
\bibitem[\v{S}a]{Shafarevich}  I. \v{S}afarevi\v{c}, {\it Le th\'eor\`eme de Torelli pour les surfaces alg\'ebriques de type K3}, Actes du Congres International des Math\'ematiciens (Nice, 1970), Gauthier-Villars, Paris, (1971), 413--417.
\bibitem[Sh]{Shim}  I.\ Shimada, {\it On elliptic K3 surfaces}, Michigan Math. J. {\bf 47} (2000), 423--446, arXiv version with the complete Table  arXiv:math/0505140.
\end{thebibliography}
\end{document}